\documentclass[12pt,a4paper,english,reqno]{amsart}
\usepackage{amsthm}
\usepackage{amssymb}
\usepackage{amscd}
\usepackage{amstext}
\usepackage[all]{xy}
\xyoption{2cell}
\UseTwocells
\usepackage{graphicx}

\theoremstyle{plain}
\newtheorem{thm}{Theorem}[section]
\newtheorem{prp}[thm]{Proposition}
\newtheorem{lem}[thm]{Lemma}
\newtheorem{cor}[thm]{Corollary}
\newtheorem{clm}{Claim}

\newtheorem*{thm-nn}{Theorem}
\newtheorem*{prp-nn}{Proposition}
\newtheorem*{lem-nn}{Lemma}
\newtheorem*{cor-nn}{Corollary}
\newtheorem*{clm-nn}{Claim}
\newtheorem*{cnj-nn}{Conjecture}
\newtheorem*{prb-nn}{Problem}

\theoremstyle{definition}
\newtheorem{dfn}[thm]{Definition}
\newtheorem{exm}[thm]{Example}

\newtheorem*{dfn-nn}{Definition}

\newtheorem{rmk}[thm]{Remark}

\newcommand{\xyR}[1]{%
\xydef@\xymatrixrowsep@{#1}}

\newcommand{\xyC}[1]{%
\xydef@\xymatrixcolsep@{#1}}

\def\al{\alpha}
\def\be{\beta}
\def\ga{\gamma}
\def\de{\delta}
\def\ep{\varepsilon}
\def\ze{\zeta}
\def\et{\eta}
\def\th{\theta}

\def\la{\lambda}
\def\ro{\rho}
\def\si{\sigma}

\def\ph{\phi}
\def\ps{\psi}

\def\De{\Delta}

\def\Hom{\operatorname{Hom}}

\def\End{\operatorname{End}}

\def\rep{\operatorname{rep}}
\def\mod{\operatorname{mod}}
\def\Mod{\operatorname{Mod}}

\def\prj{\operatorname{prj}}

\def\Prj{\operatorname{Prj}}

\def\can{\operatorname{can}}
\def\Iso{\operatorname{Iso}}

\def\Kb{{\mathcal K}^{\text{\rm b}}}

\def\perf{\operatorname{per}}

\def\calA{{\mathcal A}}
\def\calB{{\mathcal B}}
\def\calC{{\mathcal C}}
\def\calD{{\mathcal D}}

\def\calK{{\mathcal K}}

\def\calT{{\mathcal T}}

\def\bbN{{\mathbb N}}
\def\bbZ{{\mathbb Z}}

\def\op{^{\mathrm{op}}} 
\def\inv{^{-1}}

\def\implies{\text{$\Rightarrow$}\ }
\def\impliedby{\text{$\Leftarrow$}\ }

\def\incl{\hookrightarrow}
\def\iso{\cong}

\def\ox{\otimes}

\def\Lox{\overset{\mathbf{L}}{\otimes}}

\def\ovl{\overline}

\def\dsm#1,#2..#3{\bigoplus_{{#1}={#2}}^{#3}}
\def\sm#1,#2..#3{\sum_{{#1}={#2}}^{#3}}

\def\id{1\kern-.25em{\text{{\rm l}}}} 
\def\smallid{1\kern-.15em{\text{{\rm l}}}}

\def\isoto{\ \raise.8ex\hbox{$^{\sim}$}\kern-.7em\hbox{$\to$}\ }

\def\ya#1{\xrightarrow{#1}}
\def\blank{\operatorname{-}}

\def\Ltimes{\overset{\mathbf{L}}{\otimes}}

\def\bg{%
\family{cmr}\size{20}{12pt}\selectfont}

\def\bigzerou{%
\smash{\lower1.7ex\hbox{\bg 0}}}

\def\repr[#1;#2;#3;#4;#5]{
\left(
\begin{matrix}#1\\#2\end{matrix}
#3
\begin{matrix}#4\\#5\end{matrix}
\right)}
\def\colim{\varinjlim}

\def\kCat{\Bbbk\text{-}\mathbf{Cat}}
\def\GCat{G\text{-}\mathbf{Cat}}

\def\kAb{\Bbbk\text{-}\mathbf{Ab}}
\def\kTri{\Bbbk\text{-}\mathbf{Tri}}

\def\bfC{\mathbf{C}}

\def\Fun{\operatorname{Fun}}
\def\Gr{\operatorname{Gr}}
\def\k{\Bbbk}
\def\Oplax{\overleftarrow{\operatorname{Oplax}}}

\def\com{\operatorname{com}}
\def\To{\Rightarrow}

\begin{document}

\title{Derived equivalences of actions of a category}

\author{Hideto Asashiba}

\begin{abstract}
Let $\Bbbk$ be a commutative ring and $I$ a category.
As a generalization of a $\Bbbk$-category with a (pseudo) action of  a group
we consider a family of $\Bbbk$-categories with a (pseudo, lax, or oplax) action of $I$,
namely an oplax functor from $I$ to the 2-category of small $\Bbbk$-categories.
We investigate derived equivalences of those oplax functors,
and establish a Morita type theorem for them.
This gives a base of investigations of derived equivalences of Grothendieck constructions
of those oplax functors.
\end{abstract}


\thanks{This work is partially supported by Grant-in-Aid for Scientific Research
(C) 21540036 from JSPS}

\maketitle

\section{Introduction}

We fix a category $I$ and a commutative ring $\Bbbk$ and
denote by $\kCat$ (resp.\ $\kAb$, $\kTri$)
the 2-category of small $\Bbbk$-categories
(resp.\ small abelian $\k$-categories, small triangulated $\k$-categories).
For a $\k$-category $\calC$ a (right) $\calC$-{\em module} is a contravariant functor from
$\calC$ to the category $\Mod \k$ of $\k$-modules,
and we denote by $\Mod \calC$ (resp.\ $\Prj \calC$, $\prj \calC$) the category of $\calC$-modules
(resp.\ projective $\calC$-modules, finitely generated projective $\calC$-modules).

A $\k$-category $\calC$ with an action of a group $G$ have been well investigated
in connection with a so-called covering technique in representation theory of algebras
(see e.g., \cite{Gab}).
The orbit category $\calC/G$ and the canonical functor $\calC \to \calC/G$ 
are naturally constructed from these data, and one studied relationships
between $\Mod \calC$ and $\Mod \calC/G$.
We brought this point of view to the derived equivalence classification problem
of algebras in \cite{Asa97}, and a main tool obtained there
was fully used in the derived equivalence classifications
in \cite{Asa99, Asa02}.  The main tool was extended in \cite{Asa11} in
the following form:

\begin{thm}
Let $G$ be a group acting on categories $\calC$ and $\calC'$.
Assume the condition
\begin{itemize}
\item[$(*)$]
There exists a $G$-stable tilting subcategory $E$ of $\Kb(\prj \calC)$
such that there is a $G$-equivariant equivalence $\calC' \to E$.
\end{itemize}
Then the orbit categories $\calC/G$ and $\calC'/G$ are derived equivalent.
\end{thm}
(In the above, $E$ is said to be $G$-{\em stable} if the set of objects in $E$ is stable
under the $G$-action on $\Kb(\prj \calC)$ induced from that on $\calC$.)
Observe that if we regard $G$ as a category with a single
object $*$, then a $G$-action on a category $\calC$ is nothing
but a functor $X : G \to \kCat$ with $X(*)=\calC$; and
the orbit category $\calC/G$ coincides with (the $\k$-linear version of) the Grothendieck
construction $\Gr(X)$ of $X$ defined in \cite{Groth}.

In a subsequent paper \cite{Asa} we will generalize this theorem to
an arbitrary category $I$ and to any {\em oplax functors} $X, X' \colon I \to \kCat$
(roughly speaking an oplax functor $X$ is a family $(X(i))_{i\in I_0}$ of $\k$-categories indexed by
the objects $i$ of $I$ with an action of $I$,
the precise definition is given in Definition \ref{dfn:oplax-fun}).
In this paper before doing it we first investigate the meaning of the condition $(*)$.
Recall the following theorem due to Rickard \cite{Rick}:

\begin{thm}\label{Rickard-thm}
For rings $R$ and $S$ the following are equivalent:
\begin{enumerate}
\item $R$ and $S$ are derived equivalent.
\item There exists a tilting complex $T$ in $\Kb(\prj R)$ such that $\End_R(T)$ is isomorphic to $S$.
\end{enumerate}
\end{thm}
Then the condition $(*)$ can be regarded as a generalized version of the condition (2).
Therefore in this paper we first give a definition of derived equivalences of oplax functors
and generalize the theorem above in the setting of oplax functors.

Recall also that if $\calC$ is a category with an action of a group $G$, then
the module category $\Mod \calC$ (resp.\ the derived category $\calD(\Mod \calC)$)
has the induced $G$-action; thus both of them are again categories with $G$-actions.
Hence for an oplax functor $X$ the ``module category'' $\Mod X$
(resp.\ the ``derived category'' $\calD(\Mod X)$) should again be a family of categories with
an $I$-action, i.e., an oplax functor from $I$ to $\kAb$ (resp.\ to $\kTri$).
An oplax functor $\Kb(\prj X)$ is also defined as an oplax subfunctor of $\calD(\Mod X)$
by the family $(\Kb(\prj X(i)))_{i\in I_0}$, which plays the same role as $\Kb(\prj R)$
in Theorem \ref{Rickard-thm}.

We need a notion of equivalences between oplax functors for two purposes:
\begin{enumerate}
\item[(a)] to generalize the statement  $(*)$; and
\item[(b)] to define a derived equivalence of oplax functors $X$, $X'$ by an existence of
an equivalence between oplax functors  $\calD(\Mod X)$ and $\calD(\Mod X')$.
\end{enumerate}
To define equivalences of objects we need notions of 1-morphisms and 2-morphisms,
thus we need a 2-categorical structure on the collection of oplax functors.
We will define a 2-category $\Oplax(I, \bfC)$ of all oplax functors
from $I$ to a 2-category $\bfC$, which can be used for both (a) and (b) of the above.
We have the following as a corollary of our main theorem (see Theorem \ref{mainthm1} for detail),
which generalizes Theorem \ref{Rickard-thm} in the field case.
\begin{thm}
Let $X, X' \in \Oplax(I, \kCat)$.
Assume that $\k$ is a field.
Then the following are equivalent.
\begin{enumerate}
\item
$X$ and $X'$ are derived equivalent.
\item
There exists a tilting oplax functor $\calT$ for $X$
such that $\calT$ and $X'$ are equivalent in $\Oplax(I, \kCat)$.
\end{enumerate}
\end{thm}

The paper is organized as follows.
In section 2 we define a 2-category $\Oplax(I, \bfC)$ of the oplax functors
from $I$ to a 2-category $\bfC$.
In section 3 we define the ``module category'' $\Mod X$ of an oplax functor
$X \colon I \to \kCat$ as an oplax functor $I \to \kAb$.
In section 4 we define the ``derived category'' $\calD(\Mod X)$ of the oplax functor $\Mod X$
as an oplax functor $I \to \kTri$.
The constructions of oplax functors $\Mod X$ and $\calD(\Mod X)$
for an oplax functor $X \colon I \to \kCat$ in sections 4 and 5 will be unified in
the subsequent paper \cite{Asa}.
In section 5 we state and prove our main result, which gives
a characterization of derived equivalences of oplax functors
by tilting oplax subfunctors.
In section 6 as an appendix we include Keller's proof of a categorical version
of the lifting theorem in \cite{Ke2}, which is used in the proof of the main result in section 5.

\section*{Acknowledgments}
Most parts of this work was done during my stay in Bielefeld
in February and September, 2010; and some final parts
in Bonn in September, 2010.
The results were announced at the seminars in the Universities
of Bielefeld, Bonn, Paris 7, and in Beijing Normal University. 
I would like to thank Claus M.\ Ringel, Henning Krause, and
Jan Schr{\"o}er for their hospitality and nice discussions.
I would especially like to thank Bernhard Keller for giving me
a helpful suggestion of the proof of the implication
(3) $\implies$(1) in Theorem \ref{mainthm1}, and for his kindness to let me
include his proof of a generalization of his lifting theorem
to categories as an appendix.

\section{The 2-category of oplax functors}
The 2-category $\GCat$ of $\k$-categories with $G$-actions for a group $G$
was generalized by D.~Tamaki in \cite{Tam} to the 2-category $\Oplax(I, \bfC)$
of oplax functors from a category $I$ to a 2-category $\bfC$
of $V$-enriched categories for a monoidal category $V$, which we use in this paper for
$V =\Mod\k$, the category of $\k$-modules.
We refer the reader to Street \cite{Str72} for the original definition of lax functors.

Throughout this section $\bfC$ is a 2-category.

\begin{dfn}
\label{dfn:oplax-fun}
(1) An {\em oplax functor} from $I$ to $\bfC$ is a triple
$(X, \et, \th)$ of data:
\begin{itemize}
\item
a quiver morphism $X\colon I \to \bfC$, where $I$ and $\bfC$ are regarded as quivers
by forgetting additional data such as 2-morphisms or compositions;
\item
a family $\et:=(\et_i)_{i\in I_0}$ of 2-morphisms $\et_i\colon X(\id_i) \Rightarrow \id_{X(i)}$ in $\bfC$
indexed by $i\in I_0$; and
\item
a family $\th:=(\th_{b,a})_{(b,a)}$ of 2-morphisms
$\th_{b,a} \colon X(ba) \Rightarrow X(b)X(a)$
in $\bfC$ indexed by $(b,a) \in \com(I):=
\{(b,a)\in I_1 \times I_1 \mid ba \text{ is defined}\}$
\end{itemize}
satisfying the axioms:

\begin{enumerate}
\item[(a)]
For each $a\colon i \to j$ in $I$ the following are commutative:
$$
\vcenter{
\xymatrix{
X(a\id_i) \ar@{=>}[r]^(.43){\th_{a,\id_i}} \ar@{=}[rd]& X(a)X(\id_i)
\ar@{=>}[d]^{X(a)\et_i}\\
& X(a)\id_{X(i)}
}}
\qquad\text{and}\qquad
\vcenter{
\xymatrix{
X(\id_j a) \ar@{=>}[r]^(.43){\th_{\id_j,a}} \ar@{=}[rd]& X(\id_j)X(a)
\ar@{=>}[d]^{\et_jX(a)}\\
& \id_{X(j)}X(a)
}}\quad;\text{ and}
$$
\item[(b)]
For each $i \ya{a}j \ya{b} k \ya{c} l$ in $I$ the following is commutative: 
$$
\xymatrix@C=3em{
X(cba) \ar@{=>}[r]^(.43){\th_{c,ba}} \ar@{=>}[d]_{\th_{cb,a}}& X(c)X(ba)
\ar@{=>}[d]^{X(c)\th_{b,a}}\\
X(cb)X(a) \ar@{=>}[r]_(.45){\th_{c,b}X(a)}& X(c)X(b)X(a).
}
$$
\end{enumerate}

(2) A {\em lax functor} from $I$ to $\bfC$ is an oplax functor
from $I$ to $\bfC^{\text{co}}$, where $\bfC^{\text{co}}$ denotes the 2-category obtained from
$\bfC$ by reversing the 2-morphisms.

(3) A {\em pseudofunctor} from $I$ to $\bfC$ is an oplax functor $(X, \et, \th)$ with
all $\et_i$ and $\th_{b,a}$ 2-isomorphisms.
\end{dfn}

\begin{rmk}
A functor from $I$ to 
$\bfC$ is an oplax functor $(X, \et, \th)$ with
all $\et_i$ and $\th_{b,a}$ identities.

\end{rmk}

\begin{dfn}
Let 
$X = (X, \et, \th)$, $X'= (X', \et', \th')$
be oplax functors from $I$ to $\bfC$.
A {\em $1$-morphism} (called a {\em left transformation}) from $X$ to $X'$
is a pair $(F, \ps)$ of data
\begin{itemize}
\item
a family $F:=(F(i))_{i\in I_0}$ of 1-morphisms $F(i)\colon X(i) \to X'(i)$
in $\bfC$ indexed by $i \in I_0$; and
\item
a family $\ps:=(\ps(a))_{a\in I_1}$ of 2-morphisms
$\ps(a)\colon X'(a)F(i) \Rightarrow F(j)X(a)$ in $\bfC$
indexed by $a\colon i \to j$ in $I_1$:
$$
\xymatrix{
X(i) & X'(i)\\
X(j) & X'(j)
\ar_{X(a)} "1,1"; "2,1"
\ar^{X'(a)} "1,2"; "2,2"
\ar^{F(i)} "1,1"; "1,2"
\ar_{F(j)} "2,1"; "2,2"
\ar@{=>}_{\ps(a)} "1,2"; "2,1"
}
$$
\end{itemize}
satisfying the axioms
\begin{enumerate}
\item[(a)]
For each $i \in I_0$ the following is commutative:
$$
\vcenter{
\xymatrix{
X'(\id_i)F(i) & F(i)X(\id_i)\\
\id_{X'(i)}F(i) & F(i)\id_{X(i)}
\ar@{=>}^{\ps(\smallid_i)} "1,1"; "1,2"
\ar@{=} "2,1"; "2,2"
\ar@{=>}_{\et'_iF(i)} "1,1"; "2,1"
\ar@{=>}^{F(i)\et_i} "1,2"; "2,2"
}}\quad;\text{ and}
$$
\item[(b)]
For each $i \ya{a} j \ya{b} k$ in $I$ the following is commutative:
$$
\xymatrix@C=4pc{
X'(ba)F(i) & X'(b)X'(a)F(i) & X'(b)F(j)X(a)\\
F(k)X(ba) & & F(k)X(b)X(a).
\ar@{=>}^{\th'_{b,a}F(i)} "1,1"; "1,2"
\ar@{=>}^{X'(b)\ps(a)} "1,2"; "1,3"
\ar@{=>}_{F(k)\,\th_{b,a}} "2,1"; "2,3"
\ar@{=>}_{\ps(ba)} "1,1"; "2,1"
\ar@{=>}^{\ps(b)X(a)} "1,3"; "2,3"
}
$$
\end{enumerate}
\end{dfn}

\begin{dfn}
Let 
$X = (X, \et, \th)$, $X'= (X', \et', \th')$
be oplax functors from $I$ to $\bfC$, and
$(F, \ps)$, $(F', \ps')$ 1-morphisms from $X$ to $X'$.
A {2-morphism} from $(F, \ps)$ to $(F', \ps')$ is a
family $\ze= (\ze(i))_{i\in I_0}$ of 2-morphisms
$\ze(i)\colon F(i) \Rightarrow F'(i)$ in $\bfC$
indexed by $i \in I_0$
such that the following is commutative for each $a\colon i \to j$ in $I$:
$$
\xymatrix@C=4pc{
X'(a)F(i) & X'(a)F'(i)\\
F(j)X(a) & F'(j)X(a).
\ar@{=>}^{X'(a)\ze(i)} "1,1"; "1,2"
\ar@{=>}^{\ze(j)X(a)} "2,1"; "2,2"
\ar@{=>}_{\ps(a)} "1,1"; "2,1"
\ar@{=>}^{\ps'(a)} "1,2"; "2,2"
}
$$
\end{dfn}

\begin{dfn}
Let 
$X = (X, \et, \th)$, $X'= (X', \et', \th')$
and $X''= (X'', \et'', \th'')$
be oplax functors from $I$ to $\bfC$, and
let $(F, \ps)\colon X \to X'$, $(F', \ps')\colon X' \to X''$
be 1-morphisms.
Then the composite $(F', \ps')(F, \ps)$ of $(F, \ps)$ and
$(F', \ps')$ is a 1-morphism from $X$ to $X''$ defined by
$$
(F', \ps')(F, \ps):= (F'F, \ps'\circ\ps),
$$
where $F'F:=((F'(i)F(i))_{i\in I_0}$ and for each $a\colon i \to j$ in $I$,
$
(\ps'\circ\ps)(a):= F'(j)\ps(a)\circ \ps'(a)F(i)
$
is the pasting of the diagram
$$
\xymatrix@C=4pc{
X(i) & X'(i) & X''(i)\\
X(j) & X'(j) & X''(j).
\ar_{X(a)} "1,1"; "2,1"
\ar_{X'(a)} "1,2"; "2,2"
\ar^{F(i)} "1,1"; "1,2"
\ar_{F(j)} "2,1"; "2,2"
\ar@{=>}_{\ps(a)} "1,2"; "2,1"
\ar^{X''(a)} "1,3"; "2,3"
\ar^{F'(i)} "1,2"; "1,3"
\ar_{F'(j)} "2,2"; "2,3"
\ar@{=>}_{\ps'(a)} "1,3"; "2,2"
}
$$
\end{dfn}

The following is straightforward to verify.

\begin{prp}
Oplax functors $I \to \bfC$,
$1$-morphisms between them, and $2$-morphisms between
$1$-morphisms $($defined above$)$ define a $2$-category,
which we denote by $\Oplax(I, \bfC)$.
\end{prp}

In the rest of this section we give a way to construct oplax functors $I \to\bfC$.
First we recall the notion of comonads in 
$\bfC$.

\begin{dfn}
Let 
$\calC \in \bfC_0$.
A {\em comonad} on $\calC$ is a triple $(E, \si, \de)$
consisting of a 1-morphism $E \colon \calC \to \calC$ and
2-morphisms $\si \colon E \To \id_{\calC}$ and
$\de \colon E \To E^2$  in $\bfC$ such that the following diagrams commute:
$$
\vcenter{
\xymatrix{
E & E^2\\
E^2 & E
\ar@{=>}^{\de}"1,1";"1,2"
\ar@{=>}_{\si E}"2,1";"2,2"
\ar@{=>}_{\de}"1,1";"2,1"
\ar@{=>}^{E\si}"1,2";"2,2"
\ar@{=}"1,1";"2,2"
}}\quad\text{and}\quad
\vcenter{
\xymatrix{
E & E^2 \\
E^2 & E^3.
\ar@{=>}^{\de}"1,1";"1,2"
\ar@{=>}_{\de E}"2,1";"2,2"
\ar@{=>}_{\de}"1,1";"2,1"
\ar@{=>}^{E\de}"1,2";"2,2"
}
}
$$
\end{dfn}

\begin{rmk}\label{rmk:comonad}
(1) Any adjunction $(L, R, \et, \ep) \colon \calD \to \calC$ in $\bfC$
(i.e., $L \colon \calD \to \calC$ is a left adjoint to $R \colon \calC \to \calD$
with a unit $\et$ and a counit $\ep$) yields a comonad $(LR, \ep, L\et R)$
on $\calC$.

(2) Let $\calC \in \bfC_0$.  If $E \colon \calC \to \calC$ is an idempotent
(i.e., if $E^2 = E$), then a 2-morphism $\si \colon E \To \id_{\calC}$
with $\si E = \id_E = E\si$ gives a comonad $(E, \si, \id_E)$ on $\calC$.

(3) A comonad $(E, \si, \de)$ on $\calC$ is nothing but an oplax functor
$(X, \si, \de) \colon 1 \to \bfC$ with $X(1) = \calC$ and $X(\id_1) = E$,
where $1$ denotes the category with a single object $1$ with
a single morphism $\id_1$.
\end{rmk}

The following gives a way to construct an oplax functor $I \to \bfC$
using a comonad on an object $\calC$ in $\bfC$.

\begin{lem}\label{comonad-oplax}
Let $\calC \in \bfC_0$ and let $(E, \si, \de)$ be a comonad on $\calC$.
Then for any category $I$ we can construct an oplax functor
$\De(\calC, E, \si, \de):= (X, \et, \th) \colon I \to \bfC$ as follows:
\begin{itemize}
\item
$X \colon I \to \bfC$ is a quiver morphism defined by $X(i):= \calC$ for all $i \in I_0$
and $X(a):= E$ for all $a \in I_1$;
\item
$\et = (\et_i)_{i \in I_0}$, $\et_i:= \si\colon X(\id_i) \To \id_{X(i)}$ for all $i \in I_0$; and
\item
$\th =(\th_{b,a})_{(b,a)\in \com(I)}$, $\th_{b,a}:= \de \colon X(ba) \To X(b)X(a)$
for all $(b,a) \in \com(I)$.
\end{itemize}
\end{lem}

\begin{proof}
Straightforward.
\end{proof}

\begin{rmk}
Note that for any $\calC \in \bfC$ the triple
$(\id_{\calC}, \id_{\id_{\calC}}, \id_{\id_{\calC}})$ is a comonad.
Then in the above $\De(\calC, \id_{\calC}, \id_{\id_{\calC}}, \id_{\id_{\calC}})$
is just the usual diagonal functor $\De(\calC) \colon I \to \bfC$.
\end{rmk}

\begin{exm}
Using Lemma \ref{comonad-oplax}
we give a tiny example of an oplax functor $I \to \bfC$ that is not a pseudofunctor.
Here we consider the case that $\k$ is a field and $\bfC = \kCat$.
Let  $\calC$ be the path $\k$-category given by the quiver with a relation
$(\xymatrix{x \ar@/^/[r]^{\al}& y\ar@/^/[l]^{\be}}, \al\be = \id_y)$.
Using Remark \ref{rmk:comonad}(2) we define a comonad on $\calC$.
First, define a functor $E \colon \calC \to \calC$ by
setting $E(z):= y$ for all $z \in \calC_0$ and
$E(\ga):= \id_y$ for all $\ga \in \{\id_x, \id_y, \al, \be\}$.
Then obviously $E$ is an idempotent.
Second, define a natural transformation $\si \colon E \To \id_{\calC}$ by setting
$\si x:= \be \colon y \to x$ and $\si y:= \id_y \colon y \to y$.
Then by the relation $\al\be = \id_y$ it is easy to verify that $\si$ is
a natural transformation, and by definition it is obvious that
$\si E = \id_E = E\si$.
Hence we have a comonad $(E, \si, \id_E)$ on $\calC$.
Then for any category $I$ we have an oplax functor
$\De(\calC, E, \si, \id_E)\colon I \to \kCat$,
which is not a pseudofunctor because $\si x = \be$ is not an isomorphism in $\calC$.
\end{exm}

\section{The Module oplax functor}

Let $X\colon I \to \kCat$ be an oplax functor.
In this section we define
the ``module category'' $\Mod X$ of $X$ as an oplax functor $I \to \kCat$.
Recall that the {\em module category} $\Mod \calC$ of
a category $\calC \in \kCat$ is defined to be the functor category
$\kCat(\calC\op, \Mod \k)$, where $\Mod \k$ denotes the
category of $\k$-modules.
As is stated in Proposition \ref{Mod'} the composite $\Mod' \circ X$ turns out to be
a contravariant lax functor $I \to \kAb$.
When $X$ is a group action, namely when $I$ is a group $G$ and $X \colon G \to \kCat$
is a functor, the usual module category $\Mod X$ with a $G$-action of $X$
was defined to be the composite functor
$\Mod X:= \Mod' \circ X \circ i$, where $i \colon G  \to G$ is the group anti-isomorphism
defined by $x \mapsto x\inv$ for all $x \in G$.
In this way we can change $\Mod' \circ X$ to a covariant one.
But in general we cannot assume the existence of such an isomorphism $i$.
Regarding $(\Mod'\circ X)(a\inv)$ as a left adjoint to $(\Mod'\circ X)(a)$
for each $a \in G$ in the group action case,
we define $\Mod X$ by using a left adjoint $(\Mod X)(a)$ to $(\Mod'\circ X)(a)$
for each $a \in I_1$ in the general case.

\begin{dfn}
Let $X = (X, \et, \th) \in \Oplax(I, \kCat)$.
We define a lax functor
$\Mod'X = (\Mod'X, \Mod'\et, \Mod'\th) \colon I\op \to \kCat$ as
follows.
\begin{itemize}
\item
For each $i \in I_0$, $(\Mod'X)(i):=\Mod (X(i))$.
\item
For each $a \colon i \to j$ in $I$,
$(\Mod'X)(a):=(\text{-})\circ X(a) \colon (\Mod'X)(j) \to (\Mod'X)(i)$,
the restriction functor.
Namely, each $f \colon M \to N$ in $(\Mod'X)(j)$ is sent
by $(\Mod'X)(a)$ to $f\circ X(a) \colon M\circ X(a) \to N\circ X(a)$
in $(\Mod'X)(i)$, where
$$
\left\{
\begin{aligned}
(M\circ X(a))(x)&:=M(X(a)x) \text{ and}\\
(f\circ X(a))(x)&:=f(X(a)x)
\end{aligned}
\right.
$$
 for all $x \in X(i)_0$.
\item
For each $i \in I_0$, $(\Mod'\et)_i \colon \id_{(\Mod'X)(i)}
\Rightarrow (\Mod'X)(\id_i)$ is defined by
$$
((\Mod'\et)_iM)(x):=M(\et_ix) \colon M(x) \to M(X(\id_i)x)
$$
for all $M \in (\Mod'X)(i)$ and $x\in X(i)_0$.
\item
For each $i\ya{a} j \ya{b} k$ in $I$,
\\
$(\Mod'\th)_{b,a} \colon (\Mod'X)(a)\circ(\Mod'X)(b)
\Rightarrow (\Mod'X)(ba)$
is defined by
$$
((\Mod'\th)_{b,a}M)(x):= M(\th_{b,a}(x))
\colon M(X(b)(X(a)x)) \to M(X(ba)x)
$$
for all $M\in (\Mod'X)(k)$ and $x\in X(i)_0$.
\end{itemize}
\end{dfn}

\begin{prp}\label{Mod'}
In the above $\Mod'X$ is well-defined as a lax functor $I\op \to \kCat$.
\end{prp}

\begin{proof}
It is straightforward to check that both $(\Mod'\et)_i$ and
$(\Mod'\th)_{b,a}$ are natural transformations for all $i\in I_0$
and $i\ya{a} j \ya{b} k$ in $I$.

Each axiom for $\Mod'X$ to be a lax functor at a module $M$
follows from the corresponding axiom for $X$ to be
an oplax functor by applying $M$.
\end{proof}

\begin{prp}
\label{adjoint-oplax}
Let $Y'=(Y',\et', \th')\colon I\op \to \kCat$ be a lax functor.
Assume that for each $a\colon i\to j$ in $I$
there exists a left adjoint $Y(a)\colon Y'(i) \to Y'(j)$
to $Y'(a)\colon Y'(j) \to Y'(i)$ with
a unit $\ep_a\colon \id_{Y'(i)} \Rightarrow Y'(a)Y(a)$ and
a counit $\ze_a\colon Y(a)Y'(a) \Rightarrow \id_{Y'(j)}$.
We set $Y(i):=Y'(i)$ for each $i\in I_0$ to define
a quiver morphism $Y\colon I \to \kCat$.
If we define $\et:=(\et_i)_{i\in I_0}$ and
$\th:=(\th_{b,a})_{b,a}$ as follows,
then $Y:=(Y, \et, \th)\colon I \to \kCat$
turns out to be an oplax functor.

$(1)$ For each $i\in I_0$ define a natural transformation
$\et_i\colon Y(\id_i)\Rightarrow \id_{Y(i)}$
as the composite
$$
\xymatrix@1@C=2.8em{
Y(\id_i) \ar@{=>}[r]^(.38){Y(\smallid_i)\et'_i}&
Y(\id_i)Y'(\id_i)\ar@{=>}[r]^(.5){\ze_{\id_i}}&
\id_{Y'(i)}=\id_{Y(i)}.
}
$$

$(2)$ For each pair of composable morphisms
$i\ya{a}j\ya{b}k$ in $I$ define a natural transformation
$\th_{b,a}\colon Y(ba) \Rightarrow Y(b)Y(a)$  as the composite
$$
\xymatrix@C=5.5em@R=1.5ex{
Y(ba) \ar@{=>}[r]^(.4){Y(ba)\ep_a} & Y(ba)Y'(a)Y(a)
\ar@{=>}[r]^(.43){Y(ba)Y'(a)\ep_bY(a)} &
Y(ba)Y'(a)Y'(b)Y(b)Y(a)\\
\ar@{=>}[r]^(.3){Y(ba)\th'_{b,a}Y(b)Y(a)} &
Y(ba)Y'(ba)Y(b)Y(a) \ar@{=>}[r]^(.6){\ze_{ba}Y(b)Y(a)} &
Y(b)Y(a).
}
$$
\end{prp}

\begin{proof}
It is enough to verify the commutativity of the diagrams
(a$_1$), (a$_2$) below
for each $a\colon i \to j$ in $I$ and of the diagram
(b) below
for each triple $i\ya{a} j \ya{b} k \ya{c} l$ of morphisms in $I$:

\begin{equation}\tag{a$_1$}
\vcenter{
\xymatrix{
Y(a\id_i) \ar@{=>}[r]^(.43){\th_{a,\id_i}} \ar@{=}[rd]& Y(a)Y(\id_i)
\ar@{=>}[d]^{Y(a)\et_i}\\
& Y(a)\id_{Y(i)}
}}
\end{equation}

\begin{equation}\tag{a$_2$}
\vcenter{
\xymatrix{
Y(\id_j a) \ar@{=>}[r]^(.43){\th_{\id_j,a}} \ar@{=}[rd]& Y(\id_j)Y(a)
\ar@{=>}[d]^{\et_jY(a)}\\
& \id_{Y(j)}Y(a)
}}
\end{equation}

\begin{equation}\tag{b}
\vcenter{
\xymatrix@C=3em{
Y(cba) \ar@{=>}[r]^(.43){\th_{c,ba}} \ar@{=>}[d]_{\th_{cb,a}}& Y(c)Y(ba)
\ar@{=>}[d]^{Y(c)\th_{b,a}}\\
Y(cb)Y(a) \ar@{=>}[r]_(.45){\th_{c,b}Y(a)}& Y(c)Y(b)Y(a)
}}
\end{equation}

(a$_1$)
This follows from the following two commutative diagrams
(``$\sim$'' stands for a suitable functor that is uniquely
determined in the diagram):
{\scriptsize
$$
\xymatrix{
Y(a\id_i) & Y(a\id_i)Y'(\id_i)Y(\id_i)
  & Y(a\id_i)Y'(\id_i)Y'(a)Y(a)Y(\id_i) & Y(a\id_i)Y'(a\id_i)Y(a)Y(\id_i)\\
Y(a\id_i)Y'(\id_i) & Y(a\id_i)Y'(\id_i)Y(\id_i)Y'(\id_i)
  & Y(a\id_i)Y'(\id_i)Y'(a)Y(a)Y(\id_i)Y'(\id_i) & Y(a)Y(\id_i)\\
          & Y(a\id_i)Y(\id_i)
  & Y(a\id_i)Y(\id_i)Y'(a)Y(a) & Y(a)Y(\id_i)Y'(\id_i)\\
&& Y(a\id_i)Y'(a\id_i)Y(a) & Y(a)
\ar@{=>}^(.4){\sim\ep_{\id_i}}"1,1";"1,2"
\ar@{=>}^(.45){\sim\ep_{a}\sim}"1,2";"1,3"
\ar@{=>}^(.53){\sim\th'_{a,\id_i}\sim}"1,3";"1,4"
\ar@{=>}^(.38){\sim\ep_{\id_i}\sim}"2,1";"2,2"
\ar@{=>}^(.43){\sim\ep_{a}\sim}"2,2";"2,3"
\ar@{=>}^(.43){\sim\ep_{a}}"3,2";"3,3"
\ar@{=>}^(.6){\ze_{a\id_i}\sim}"4,3";"4,4"
\ar@{=>}^{\sim\et'_i}"1,1";"2,1"
\ar@{=>}^{\sim\et'_i}"1,3";"2,3"
\ar@{=>}^{\ze_{a\id_i}\sim}"1,4";"2,4"
\ar@{=>}^{\sim\ze_{\id_i}}"2,2";"3,2"
\ar@{=>}^{\sim\ze_{\id_i}}"2,3";"3,3"
\ar@{=>}^{\sim\et'_{i}}"2,4";"3,4"
\ar@{=>}^{\sim\th'_{a,\id_i}\sim}"3,3";"4,3"
\ar@{=>}^{\sim\ze_{\id_i}}"3,4";"4,4"
\ar@{=}"2,1";"3,2"
}
$$
}
and
$$
\xymatrix{
Y(a) & Y(a\id_i) & Y(a\id_i)Y'(\id_i)\\
  & Y(a\id_i)Y'(a)Y(a) & Y(a\id_i)Y'(\id_i)Y'(a)Y(a)\\
  &  & Y(a\id_i)Y'(a\id_i)Y(a).
\ar@{=}"1,1";"1,2"
\ar@{=>}^{\sim\et'_i}"1,2";"1,3"
\ar@{=>}^(.44){\sim\et'_i\sim}"2,2";"2,3"
\ar@{=>}^{\sim\ep_a}"1,2";"2,2"
\ar@{=>}^{\sim\ep_a}"1,3";"2,3"
\ar@{=>}^{\sim\th'_{a,\id_i}\sim}"2,3";"3,3"
\ar@{=}"2,2";"3,3"
\ar@{=>}^(.6){\ze_a\sim}"2,2";"1,1"
}
$$

(a$_2$) This follows similarly.

(b) Glue the following two commutative diagrams together
at the common column.

\begin{center}
{\tiny
$$
\xymatrix@C=1.6em{
Y(cba) & Y(cba)Y'(ba)Y(ba) & Y(cba)Y'(ba)Y'(c)Y(c)Y(ba)\\
Y(cba)Y'(a)Y(a) && Y(cba)Y'(ba)Y'(c)Y(c)Y(ba)Y'(a)Y(a)\\
 Y(cba)Y'(a)Y'(b)Y(b)Y(a) & Y(cba)Y'(ba)Y(ba)Y'(a)Y'(b)Y(b)Y(a)
 & Y(cba)Y'(ba)Y'(c)Y(c)Y(ba)Y'(a)Y'(b)Y(b)Y(a)\\
Y(cba)Y'(ba)Y(b)Y(a) & Y(cba)Y'(ba)Y(ba)Y'(ba)Y(b)Y(a)
 & Y(cba)Y'(ba)Y'(c)Y(c)Y(ba)Y'(ba)Y(b)Y(a)\\
& Y(cba)Y'(ba)Y(b)Y(a) & Y(cba)Y'(ba)Y'(c)Y(c)Y(b)Y(a)
\ar@{=>}^{\sim\ep_{ba}}"1,1";"1,2"
\ar@{=>}^{\sim\ep_{c}\sim}"1,2";"1,3"
\ar@{=>}^(.42){\sim\ep_{ba}\sim}"3,1";"3,2"
\ar@{=>}^(.44){\sim\ep_{c}\sim}"3,2";"3,3"
\ar@{=>}^(.42){\sim\ep_{ba}\sim}"4,1";"4,2"
\ar@{=>}^{\sim\ep_{c}\sim}"5,2";"5,3"
\ar@{=>}^{\sim\ep_{a}}"1,1";"2,1"
\ar@{=>}^{\sim\ep_{b}\sim}"2,1";"3,1"
\ar@{=>}^{\sim\th'_{b,a}\sim}"3,1";"4,1"
\ar@{=>}^{\sim\th'_{b,a}\sim}"3,2";"4,2"
\ar@{=>}^{\sim\ze_{ba}\sim}"4,2";"5,2"
\ar@{=>}^{\sim\ep_{a}}"1,3";"2,3"
\ar@{=>}^{\sim\ep_{b}\sim}"2,3";"3,3"
\ar@{=>}^{\sim\th'_{b,a}\sim}"3,3";"4,3"
\ar@{=>}^{\sim\ze_{ba}\sim}"4,3";"5,3"
\ar@{=}"4,1";"5,2"
}
$$
}
{\scriptsize
$$
\xymatrix@C=1em{
Y(cba)Y'(ba)Y'(c)Y(c)Y(ba)
 & Y(cba)Y'(cba)Y(c)Y(ba) & Y(c)Y(ba)\\
Y(cba)Y'(ba)Y'(c)Y(c)Y(ba)Y'(a)Y(a)
 && Y(c)Y(ba)Y'(a)Y(a)\\
Y(cba)Y'(ba)Y'(c)Y(c)Y(ba)Y'(a)Y'(b)Y(b)Y(a)
  && Y(c)Y(ba)Y'(a)Y'(b)Y(b)Y(a)\\
Y(cba)Y'(ba)Y'(c)Y(c)Y(ba)Y'(ba)Y(b)Y(a)
  && Y(c)Y(ba)Y'(ba)Y(b)Y(a)\\
Y(cba)Y'(ba)Y'(c)Y(c)Y(b)Y(a)
 & Y(cba)Y'(cba)Y(c)Y(b)Y(a) & Y(c)Y(b)Y(a)
\ar@{=>}^(.53){\sim\th'_{c,ba}\sim}"1,1";"1,2"
\ar@{=>}^(,6){\ze_{cba}\sim}"1,2";"1,3"
\ar@{=>}^(.52){\sim\th'_{c,ba}\sim}"5,1";"5,2"
\ar@{=>}^(.6){\ze_{cba}\sim}"5,2";"5,3"
\ar@{=>}^{\sim\ep_{a}}"1,1";"2,1"
\ar@{=>}^{\sim\ep_{b}\sim}"2,1";"3,1"
\ar@{=>}^{\sim\th'_{b,a}\sim}"3,1";"4,1"
\ar@{=>}^{\sim\ze_{ba}\sim}"4,1";"5,1"
\ar@{=>}^{\sim\ep_{a}}"1,3";"2,3"
\ar@{=>}^{\sim\ep_{b}\sim}"2,3";"3,3"
\ar@{=>}^{\sim\th'_{b,a}\sim}"3,3";"4,3"
\ar@{=>}^{\sim\ze_{ba}\sim}"4,3";"5,3"
}
$$
}
\end{center}
Then the composite of the path consisting of the top row
and the right most column gives the clockwise composite of (b).
The glued diagram and the commutative diagram
$$
\xymatrix{
Y(cba)Y'(a)Y'(b)Y(b)Y(a) & Y(cba)Y'(ba)Y(b)Y(a)\\
Y(cba)Y'(a)Y'(b)Y'(c)Y(c)Y(b)Y(a) & Y(cba)Y'(ba)Y'(c)Y(c)Y(b)Y(a)
\ar@{=>}^(.53){\sim\th'_{b,a}\sim}"1,1";"1,2"
\ar@{=>}^(.53){\sim\th'_{b,a}\sim}"2,1";"2,2"
\ar@{=>}^{\sim\ep_{c}\sim}"1,1";"2,1"
\ar@{=>}^{\sim\ep_{c}\sim}"1,2";"2,2"
}
$$
show that the clockwise composite of (b) is given by
$$
Y(c)\th_{b,a}\circ  \th_{c,ba}=
(\ze_{cba}\sim)\circ (\sim\th'_{c,ba}\sim)\circ (\sim\th'_{b,a}\sim)
\circ (\sim\ep_c\sim)\circ (\sim\ep_b\sim)\circ (\sim\ep_a).
$$
Similarly the anti-clockwise composite of (b) is given by
$$
\th_{c,b}Y(a)\circ \th_{cb,a}=
(\ze_{cba}\sim)\circ (\sim \th'_{cb,a}\sim)\circ (\sim\th'_{c,b}\sim)
\circ (\sim\ep_c\sim)\circ(\sim\ep_b\sim)\circ(\sim\ep_a).
$$
Hence they coincide because $Y'$ is a lax functor and
$$
(\sim\th'_{c,ba}\sim)\circ (\sim\th'_{b,a}\sim)
=(\sim \th'_{cb,a}\sim)\circ (\sim\th'_{c,b}\sim).
$$

\end{proof}

\begin{dfn}
Let $X=(X,\et,\th) \in \Oplax(I, \kCat)$.
We define the {\em module oplax functor}
$\Mod X \in \Oplax(I, \kAb)$ of $X$ as follows.
\begin{itemize}
\item
For each $i \in I_0$, we set
$(\Mod X)(i):= (\Mod'X)(i):= \Mod (X(i))$.
\item
For each $a \colon i \to j$ in $I$, define an $X(i)$-$X(j)$-bimodule
$\ovl{X(a)}
$
by
$$
\ovl{X(a)}(x, y):= X(j)(y, X(a)(x))
$$
for all $x \in X(i)_0$ and $y \in X(j)_0$.
Using this bimodule we define a functor
$(\Mod X)(a) \colon (\Mod X)(i) \to (\Mod X)(j)$
by $(\Mod X)(a):= \text{-}\otimes_{X(i)}\ovl{X(a)}$
that is a left adjoint to $(\Mod'X)(a)$.
\item
By applying Proposition \ref{adjoint-oplax} to the lax functor $\Mod'X$
and $(\Mod X)(a)$ ($a \in I_1$), we can define
an oplax functor
$\Mod X = (\Mod X, \Mod \et, \Mod \th) \colon I \to \kAb$.
\end{itemize}
\end{dfn}

\section{The derived oplax functor}

\begin{dfn}
Let $Y=(Y, \et, \th) \in \Oplax(I, \kAb)$.
Then we define an oplax functor $\calC(Y) =(\calC(Y), \calC(\et), \calC(\th)) \in \Oplax(I, \kCat)$
and an oplax functor $\calK(Y) =(\calK(Y), \calK(\et), \calK(\th)) \in \Oplax(I, \kTri)$
as follows.

\begin{enumerate}
\item $\calC(Y)$
\begin{itemize}
\item
For each $i \in I_0$, $\calC(Y)(i):= \calC(Y(i))$, the category of
complexes in $Y(i)$.
\item
For each $a \colon i \to j$ in $I$,
$\calC(Y)(a) \colon \calC(Y)(i) \to \calC(Y)(j)$ is a functor
defined as follows:
 \begin{itemize}
 \item
 For each $(x^p,d^p_x)_{p \in \bbZ} \in \calC(Y)(i)_0$,
 $$
 \calC(Y)(a)(x^p,d^p_x)_{p \in \bbZ}:=(Y(a)x^p, Y(a)d^p_x)_{p \in \bbZ}.
 $$
 \item
 For each
 $(f^p)_{p\in \bbZ} \colon (x^p,d^p_x)_{p \in \bbZ} \to (y^p,d^p_y)_{p \in \bbZ}$
 in $\calC(Y)(i)$,
 $$
 \calC(Y)(a)(f^p)_{p\in \bbZ} := (Y(a)f^p)_{p\in \bbZ}.
 $$
 \end{itemize}
\item
For each $i \in I_0$, $\calC(\et)_i \colon \calC(Y)(\id_i)
\Rightarrow \id_{\calC(Y)(i)}$
is defined by
$$
\calC(\et)_i(x^p,d^p_x)_{p \in \bbZ}:= (\et_i(x^p))_{p\in \bbZ}
\colon (Y(\id_i)x^p, Y(\id_i)d^p_x)_{p\in \bbZ} \to (x^p, d^p_x)_{p\in \bbZ}
$$
for all $(x^p, d^p_x)_{p\in \bbZ} \in \calC(Y)(i)_0$.
\item
For each $i \ya{a} j \ya{b} k$ in $I$,
$\calC(\th)_{b,a} \colon \calC(Y)(ba) \Rightarrow \calC(b)\calC(a)$
is defined by
$$
\calC(\th)_{b,a}(x^p, d^p_x)_{p\in \bbZ}:= (\th_{b,a}(x^p))_{p\in \bbZ}
\colon (Y(ba)x^p, Y(ba)d^p)_{p\in \bbZ} \to
(Y(b)Y(a)x^p, Y(b)Y(a)d^p)_{p\in \bbZ}
$$
for all $(x^p, d^p_x)_{p\in \bbZ} \in \calK(Y)(i)_0$.
\end{itemize}
\item
$\calK(Y)$
\ \ (underlines denote the homotopy classes below)
\begin{itemize}
\item
For each $i \in I_0$, $\calK(Y)(i):= \calK(Y(i))$, the homotopy category of
$Y(i)$.
\item
For each $a \colon i \to j$ in $I$,
$\calK(Y)(a) \colon \calK(Y)(i) \to \calK(Y)(j)$ is a functor
defined as follows:
 \begin{itemize}
 \item
 For each $(x^p,d^p_x)_{p \in \bbZ} \in \calK(Y)(i)_0$,
 $$
 \calK(Y)(a)(x^p,d^p_x)_{p \in \bbZ}:=(Y(a)x^p, Y(a)d^p_x)_{p \in \bbZ}.
 $$
 \item
 For each
 $\underline{(f^p)_{p\in \bbZ}} \colon (x^p,d^p_x)_{p \in \bbZ} \to (y^p,d^p_y)_{p \in \bbZ}$
 in $\calK(Y)(i)$,
 $$
 \calK(Y)(a)\underline{(f^p)_{p\in \bbZ}} := \underline{(Y(a)f^p)_{p\in \bbZ}}.
 $$
 \end{itemize}
\item
For each $i \in I_0$, $\calK(\et)_i \colon \calK(Y)(\id_i)
\Rightarrow \id_{\calK(Y)(i)}$
is defined by
$$
\calK(\et)_i(x^p,d^p_x)_{p \in \bbZ}:= \underline{(\et_i(x^p))_{p\in \bbZ}}
\colon (Y(\id_i)x^p, Y(\id_i)d^p_x)_{p\in \bbZ} \to (x^p, d^p_x)_{p\in \bbZ}
$$
for all $(x^p, d^p_x)_{p\in \bbZ} \in \calK(Y)(i)_0$.
\item
For each $i \ya{a} j \ya{b} k$ in $I$,
$\calK(\th)_{b,a} \colon \calK(Y)(ba) \Rightarrow \calK(b)\calK(a)$
is defined by
$$
\calK(\th)_{b,a}(x^p, d^p_x)_{p\in \bbZ}:= \underline{(\th_{b,a}(x^p))_{p\in \bbZ}}
\colon (Y(ba)x^p, Y(ba)d^p)_{p\in \bbZ} \to
(Y(b)Y(a)x^p, Y(b)Y(a)d^p)_{p\in \bbZ}
$$
for all $(x^p, d^p_x)_{p\in \bbZ} \in \calK(Y)(i)_0$.
\end{itemize}
\end{enumerate}
\end{dfn}

\begin{prp}
In the above, $\calC(Y)$ and $\calK(Y)$ are well-defined.
\end{prp}

\begin{proof}
Straightforward.
\end{proof}

\begin{exm}\label{KMod}
Let $X \in \Oplax(I, \kCat)$.
Then applying the proposition above to $Y:= \Mod X$ we obtain
the definition of $\calK(\Mod X)$.
Note that for each $a\colon i \to j$ in $I$,
$\calK(\Mod X)(i) \xrightarrow{\calK(\Mod X)(a)} \calK(\Mod X)(j)$
is equal to
$$\calK(\Mod X(i)) \xrightarrow{\text{-}\otimes_{X(i)}\ovl{X(a)}}\calK(\Mod X(j)).$$
\end{exm}

The following is obvious.

\begin{lem}
\label{oplax-subfunctor}
Let $X=(X, \et, \th) \in \Oplax(I, \kCat)$ and
$Y(i)$ be a subcategory of $X(i)$ for each $i \in I_0$.
Assume that $X(a)(Y(i)_0) \subseteq Y(j)_0$
and $X(a)(Y(i)_1) \subseteq Y(j)_1$ for each $a \colon i \to j$ in $I$.
Then we can define a functor $Y(a) \colon Y(i) \to Y(j)$ as the restriction
$Y(a):=X(a)|_{Y(i)}$ of $X(a)$ to $Y(i)$
for each $a \colon i \to j$ in $I$.
If we define
$\et':=(\et'_i)_{i\in I_0}$ and $\th':=(\th'_{b,a})$ as follows,
then $Y = (Y, \et', \th')$ turns out to be an oplax functor $I \to \kCat$:
\begin{enumerate}
\item
For each $i \in I_0$ and each $y \in Y(i)_0$, we set
$$
\et'_i(y):= \et_i(y) \colon Y(\id_i)y \to y;
$$
\item
For each $i \ya{a} j \ya{b} k$ in $I$ and $y \in Y(i)_0$, we set
$$
\th'_{b,a}(y):= \th_{b,a}(y) \colon Y(ba)y \to Y(b)Y(a)y.
$$
\end{enumerate}
This $Y$ is called the {\em oplax subfunctor} of $X$ defined by
$Y(i)$ $(i\in I_0)$. \qed
\end{lem}

\begin{dfn}
Let $X=(X, \et, \th) \in \Oplax(I, \kCat)$.
Using the lemma above we define the oplax subfunctor $\prj X$ of $\Mod X$ by $\prj (X(i))$ $(i\in I_0)$ and the oplax subfunctor $\Kb(\prj X)$ of $\calK(\Mod X)$
by $\Kb(\prj (X(i)))$ $(i\in I_0)$.
\end{dfn}

For each $\k$-category $\calC$ we denote by $\calK_p(\Mod \calC)$ the full subcategory
of $\calK(\Mod \calC)$ consisting of 
objects $M$ such that $\calK(\Mod \calC)(M, A) = 0$ for all acyclic objects $A$
in $\calK(\Mod \calC)$.

\begin{lem}
Let $X=(X, \et, \th) \in \Oplax(I, \kCat)$.
Then for each $a \colon i \to j$ in $I$,
$\calK(\Mod X)(a)(\calK_p(\Mod X(i)) \subseteq \calK_p(\Mod X(j))$.
\end{lem}

\begin{proof}
Let $A$ be an acyclic complex in $\calK(\Mod X(j))$ and
let $P \in \calK_p(\Mod X(i))$.
Note that
$\calK(\Mod X)(a)$ is a left adjoint to $\calK(\Mod' X)(a)$
by the definition of $\calK(\Mod X)$ in Example \ref{KMod}.
Then we have
$\calK(\Mod X(j))(\calK(\Mod X)(a)(P), A)\iso \calK(\Mod X(i))(P, A\circ X(a))=0$
because $A\circ X(a)$ is acyclic in $\calK(\Mod X(i))$.
Thus $\calK(\Mod X)(a)(P) \in \calK_p(\Mod X(j))$.
\end{proof}

This enables us to define the following by Lemma \ref{oplax-subfunctor}.

\begin{dfn}
Let $X=(X, \et, \th) \in \Oplax(I, \kCat)$.
Then we define an oplax functor $\calK_p(\Mod X) \in \Oplax(I, \kTri)$
as the oplax subfunctor of $\calK(\Mod X)$ defined by $\calK_p(\Mod X(i))$
($i \in I_0$).
\end{dfn}

\begin{lem}
\label{adjoint-oplax}
Let $X=(X, \et, \th) \in \Oplax(I, \kCat)$.
For each $i \in I_0$, let
$Y(i) \in \kCat_0$ and
assume that we have an adjoint pair $\xymatrix@C=1em{L_i \ar@{-|}[r]& R_i}$
of functors $L_i \colon X(i) \to Y(i)$
and $R_i \colon Y(i) \to X(i)$ 
with a unit $\ep_i \colon \id_{X(i)} \Rightarrow R_iL_i$ and
a counit $\ze_i \colon L_iR_i \Rightarrow \id_{Y(i)}$.

Define $(Y, \et', \th')$ as follows.
\begin{itemize}
\item
Define
$Y(a) \colon Y(i) \to Y(j)$ as $Y(a):= L_jX(a)R_i$
for each $a \colon i \to j$ in $I$.
\item
For each $i \in I_0$, define $\et'_i \colon Y(\id_i) \Rightarrow \id_{Y(i)}$
as the composite
$$
\xymatrix@1{Y(\id_i)=L_iX(\id_i)R_i \ar@{=>}[r]^(.7){L_i\et_iR_i}&
L_iR_i\ar@{=>}[r]^{\ze_i}&\id_{Y(i)}
}
$$
and set $\et':=(\et'_i)_{i\in I_0}$;
\item
For each $i \ya{a} j \ya{b} k$ in $I$,
define $\th'_{b,a} \colon Y(ba) \Rightarrow Y(b)Y(a)$
as the composite
$$
\xymatrix@C=2em@R=1ex{
Y(ba)=L_kX(ba)R_i \ar@{=>}[r]^{L_k\th_{b,a}R_i}& L_kX(b)X(a)R_i\\
\ar@{=>}[r]^(.25){L_kX(b)\ep_jX(a)R_j}& L_kX(b)R_jL_jX(a)R_i=Y(b)Y(a).
}
$$
and set $\th':=(\th'_{b,a})$.
\end{itemize}
Then
\begin{enumerate}
\item
$Y = (Y, \et', \th') \in \Oplax(I, \kCat)$.
We say that this $Y$ is the oplax functor {\em induced from} $X$ {\em by}
adjoint pairs
$L_i, R_i$ $(i\in I_0)$.
\item
The family $(R_i)_{i\in I_0}$ is extended to a morphism
$$
R=(R, \ph^R) \colon Y \to X \quad\text{in } \Oplax(I,\kCat)
$$ by defining
$R(i):=R_i$ for all $i\in I_0$ and
$\ph^R(a):= \ep_j X(a)R_i$ for all $a\colon i \to j$ in $I$.
\item
Assume further that $\ep_i$ is an isomorphism for each $i \in I_0$.
Then
\begin{enumerate}
\item
$\ph^R(a)$ is an isomorphism for each $a \in I_1$, i.e., $R$ is
$I$-equivariant; and
\item
the family $(L_i)_{i\in I_0}$ is extended to an $I$-equivariant morphism
$$L=(L, \ph^L)\colon X \to Y
\quad\text{in } \Oplax(I, \kCat)
$$ by defining
$L(i):=L_i$ for all $i \in I_0$ and
$\ph^L(a):=L_j X(a)\ep_i\inv$ for all $a\colon i \to j$ in $I$.
\end{enumerate}
\end{enumerate}
\end{lem}

\begin{proof}
(1) For each $a \colon i \to j$ in $I$,
the axiom (a$_1$) follows from the following commutative diagram:
$$
\xymatrix{
L_jX(a\id_i)R_i & L_jX(a)X(\id_i)R_i & L_jX(a)R_iL_iX(\id_i)R_i\\
&L_jX(a)R_i & L_jX(a)R_iL_iR_i\\
&& L_jX(a)R_i
\ar@{=>}^(.45){\sim\th_{a,\id_i}\sim}"1,1";"1,2"
\ar@{=>}^(.45){\sim\ep_{i}\sim}"1,2";"1,3"
\ar@{=>}^{\sim\ep_{i}\sim}"2,2";"2,3"
\ar@{=>}^{\sim\et_{i}\sim}"1,2";"2,2"
\ar@{=>}^{\sim\et_{i}\sim}"1,3";"2,3"
\ar@{=>}^{\sim\ze_{i}}"2,3";"3,3"
\ar@{=}"1,1";"2,2"
\ar@{=}"2,2";"3,3"
}
$$

The axiom (a$_2$) follows similarly.

For each $i\ya{a} j \ya{b} k \ya{c} l$ in $I$,
the axiom (b) follows from the following commutative diagram:
$$
\xymatrix{
L_lX(cba)R_i & L_lX(c)X(ba)R_i & L_lX(c)R_kL_kX(ba)R_i\\
L_lX(cb)X(a)R_i & L_lX(c)X(b)X(a)R_i & L_lX(c)R_kL_kX(b)X(a)R_i\\
L_lX(cb)R_jL_jX(a)R_i & L_lX(c)X(b)R_jL_jX(a)R_i & L_lX(c)R_kL_kX(b)R_jL_jX(a)R_i.
\ar@{=>}^(.45){\sim\th_{c,ba}\sim}"1,1";"1,2"
\ar@{=>}^(.45){\sim\ep_{k}\sim}"1,2";"1,3"
\ar@{=>}^(.45){\sim\th_{c,b}\sim}"2,1";"2,2"
\ar@{=>}^(.45){\sim\ep_{k}\sim}"2,2";"2,3"
\ar@{=>}^(.45){\sim\th_{c,b}\sim}"3,1";"3,2"
\ar@{=>}^(.45){\sim\ep_{k}\sim}"3,2";"3,3"
\ar@{=>}_{\sim\th_{cb,a}\sim}"1,1";"2,1"
\ar@{=>}^{\sim\th_{b,a}\sim}"1,2";"2,2"
\ar@{=>}^{\sim\th_{b,a}\sim}"1,3";"2,3"
\ar@{=>}_(.45){\sim\ep_{j}\sim}"2,1";"3,1"
\ar@{=>}^(.45){\sim\ep_{j}\sim}"2,2";"3,2"
\ar@{=>}^(.45){\sim\ep_{j}\sim}"2,3";"3,3"
}
$$

(2) and (3)  Straightforward.
\end{proof}

\begin{dfn}
Let $X=(X, \et, \th) \in \Oplax(I, \kCat)$.
\begin{enumerate}
\item
For each $i \in I_0$ let $L_i$ be the composite
$\calK_p(\Mod X(i)) \incl \calK(\Mod X(i)) \to \calD(\Mod X(i))$
of the embedding and the quotient functor.
Then it is well-known that $L_i$ is a triangle equivalence with
a quasi-inverse
$R_i \colon \calD(\Mod X(i)) \to \calK_p(\Mod X(i))$.
Then we define the {\em derived oplax functor}
$$
\calD(\Mod X) \in \Oplax(I, \kTri)
$$
of $X$ as the oplax functor induced from
the oplax functor $\calK_p(\Mod X)$
by the adjoint pairs $L_i, R_i$ ($i \in I_0$).
\item
We define an oplax functor $\perf X \in \Oplax(I, \kTri)$
as the oplax subfunctor of $\calD(\Mod X)$
defined by the subcategories $\perf X(i)$ of $\calD(\Mod X(i))$ ($i \in I_0$)
consisting of the {\em perfect complexes} $M$, i.e.,
the objects $M$ such that $\calD(\Mod X(i))(M, \blank)$
commutes with arbitrary (set-indexed) direct sums.
\end{enumerate}
\end{dfn}

\begin{rmk}
\label{D-form}
In the above, we have the following by definition.
\begin{enumerate}
\item
For each $a\colon i \to j$ in $I$,
$\calD(\Mod X)(i) \xrightarrow{\calD(\Mod X)(a)} \calD(\Mod X)(j)$
is equal to
$$\calD(\Mod X(i)) \xrightarrow{\text{-}\Lox_{X(i)}\ovl{X(a)}}\calD(\Mod X(j)).$$
\item
$\perf X$ and $\Kb(\prj X)$ are equivalent
in the 2-category $\Oplax(I, \kTri)$ in the sense recalled in Section 5.
\end{enumerate}
\end{rmk}

In the subsequent paper \cite{Asa} we will give a unified way to define oplax functors
$\Mod X$, $\calK(\Mod X)$, \ldots, $\calD(\Mod X)$ for $X \in \Oplax(I, \kCat)$
using composites of oplax functors and
pseudofunctors between 2-categories.
\section{Derived equivalences of oplax functors}

Recall that a 1-morphism $f \colon x \to y$ in a 2-category $\bfC$
is called an {\em equivalence} if there exist 1-morphism $g \colon y \to x$
in $\bfC$ such that there exist 2-isomorphisms
$gf \Rightarrow \id_x$ and $fg \Rightarrow \id_y$ in $\bfC$;
and that two objects $x$ and $y$ in $\bfC$ are called
{\em equivalent} in $\bfC$ if there exists an equivalence $f \colon x \to y$.

\begin{lem}
\label{oplax-eq}
Let $\bfC$ be a $2$-category and $(F, \ps) \colon X \to X'$
a $1$-morphism in the $2$-category $\Oplax(I, \bfC)$.
Then $(F, \ps)$ is an equivalence in $\Oplax(I, \bfC)$
if and only if
\begin{enumerate}
\item
For each $i \in I_0$, $F(i)$
is an equivalence in $\bfC$; and
\item
For each $a \in I_1$, $\ps(a)$ is a $2$-isomorphism in $\bfC$
$($namely, $(F,\ps)$ is $I$-equivariant$)$.
\end{enumerate}
\end{lem}

\begin{proof}
Set $X = (X, \et, \th)$ and $X' = (X', \et', \th')$.

(\implies). 
Assume that $(F, \ps)$ is an equivalence in $\Oplax(I, \bfC)$.
Then there exists a 1-morphism $(E, \ph) \colon X' \to X$ and 2-isomorphisms
$$
\begin{aligned}
\ze \colon (\id_X, (\id_{X(a)})_a) &\Longrightarrow (E,\ph)\circ (F,\ps) 
= (EF, (E(j)\ps(a)\circ \ph(a)F(i))_{a:i\to j})\text{ and}\\
\ze' \colon (\id_X', (\id_{X'(a)})_a) &\Longrightarrow (F,\ps)\circ (E,\ph)
= (FE, (F(j)\ph(a)\circ \ps(a)E(i))_{a:i\to j})
\end{aligned}
$$
in $\Oplax(I, \bfC)$.
Then the assumption first shows that for each $i \in I_0$ we have 2-isomorphisms
$\ze(i) \colon \id_{X(i)} \Longrightarrow E(i)F(i)$ and
$\ze'(i) \colon \id_{X'(i)} \Longrightarrow F(i)E(i)$ in $\bfC$.
Thus $F(i)$ is an equivalence in $\bfC$ with a quasi-inverse $E(i)$.

Next the assumption shows that for each $a \colon i \to j$ in $I$
we have the following commutative diagrams of 2-morphisms in $\bfC$:
\begin{equation}
\label{2-iso}
\xymatrix{
X(a)\id_{X(i)} & X(a)E(i)F(i)\\
\id_{X(j)}X(a) & E(j)F(j)X(a)
\ar@{=>}^(.45){X(a)\ze(i)}_(.45){\sim}"1,1";"1,2"
\ar@{=>}_(.45){\ze(j)X(a)}^(.45){\sim}"2,1";"2,2"
\ar@{=>}_(.45){\id_{X(a)}}"1,1";"2,1"
\ar@{=>}^(.45){E(j)\ps(a)\circ \ph(a)F(i)}"1,2";"2,2"
}
\quad
\xymatrix{
X'(a)\id_{X'(i)} & X'(a)F(i)E(i)\\
\id_{X'(j)}X'(a) & F(j)E(j)X'(a)
\ar@{=>}^(.45){X'(a)\ze'(i)}_(.45){\sim}"1,1";"1,2"
\ar@{=>}_(.45){\ze'(j)X'(a)}^(.45){\sim}"2,1";"2,2"
\ar@{=>}_(.45){\id_{X'(a)}}"1,1";"2,1"
\ar@{=>}^(.45){F(j)\ph(a)\circ \ps(a)E(i)}"1,2";"2,2"
}
\end{equation}
Hence both $E(j)\ps(a)\circ \ph(a)F(i)$ and $F(j)\ph(a)\circ \ps(a)E(i)$
are 2-isomorphisms in $\bfC$, and hence so are
both $F(j)E(j)\ps(a)\circ F(j)\ph(a)F(i)$
and $F(j)\ph(a)F(i)\circ \ps(a)E(i)F(i)$.
Thus $F(j)E(j)\ps(a)$ is a 2-retraction and
$\ps(a)E(i)F(i)$ is a 2-section.
Hence $\ps(a)$ is a 2-isomorphism in $\bfC$
because both $\ze(i)$ and $\ze'(j)$
are 2-isomorphisms.

(\impliedby).
Conversely, assume that $(F, \ps)$ satisfies the conditions (1) and (2).
Then by (1), for each $i \in I_0$ there exists a quasi-inverse $E(i)$
of $F(i)$, thus there exist 2-isomorphisms
$\ze_i\colon \id_{X(i)} \Longrightarrow E(i)F(i)$ and
$\ze'_i \colon \id_{X'(i)} \Longrightarrow F(i)E(i)$
satisfying the following equations:
\begin{align}
E(i)\ze'_i &= \ze_i E(i) \label{form-b'}\\
F(i)\ze_i &= \ze'_iF(i) \label{form-a'}
\end{align}

By (2) we can construct a $\ph:= (\ph(a))_{a\in I_1}$
by the following
commutative diagram for each $a\colon i \to j$ in $I$:
$$
\xymatrix@C=15ex{
X(a)E(i) & E(j)X'(a)\\
E(j)F(j)X(a)E(i) & E(j)X'(a)F(i)E(i).
\ar@{=>}^{\ph(a)}"1,1";"1,2"
\ar@{=>}_{E(j)\ps(a)\inv E(i)}"2,1";"2,2"
\ar@{=>}_{\ze(j)X(a)E(i)}"1,1";"2,1"
\ar@{=>}_{E(j)X'(a)\ze'(i)\inv}"2,2";"1,2"
}
$$
It is enough to show the following
\begin{align}\label{oplax-mor}
(E, \ph)&\colon (X',\et',\th')\to (X, \et, \th)
\text{ is in $\Oplax(I,\kCat)$;}\\
\label{z-2iso}
\ze:= (\ze(i))_{i\in I_0}&\colon (\id_X, (\id_{X(a)})_{a})
\Longrightarrow (E,\ph)\circ (F, \ps)
\text{ is a 2-isomorphism; and}\\
\label{z'-2iso}
\ze':=(\ze'(i))_{i\in I_0}&\colon (\id_{X'}, (\id_{X'(a)})_{a})
\Longrightarrow (F, \ps)\circ (E,\ph)
\text{ is a 2-isomorphism.}
\end{align}
Using \eqref{form-b'} it is not hard to verify the
statement \eqref{oplax-mor}.
The statements \eqref{z-2iso} and \eqref{z'-2iso}
are equivalent to the commutativity of the left diagram
and of the right diagram in \eqref{2-iso},
respectively, and both follow from \eqref{form-a'}.
\end{proof}

\begin{dfn}
Let $X, X' \in \Oplax(I, \kCat)$.
Then $X$ and $X'$ are said to be {\em derived equivalent} if
$\calD(\Mod X)$ and $\calD(\Mod X')$ are equivalent
in the 2-category $\Oplax(I, \kTri)$.
\end{dfn}

By Lemma \ref{oplax-eq} we obtain the following.

\begin{prp}
\label{der-eq-criterion}
Let $X, X' \in \Oplax(I, \kCat)$.
Then $X$ and $X'$ are derived equivalent if and only if
there exists a 1-morphism
$(F, \ps) \colon \calD(\Mod X) \to \calD(\Mod X')$ in $\Oplax(I, \kTri)$ such that
\begin{enumerate}
\item
For each $i \in I_0$, $F(i)$
is a triangle equivalence; and
\item
For each $a \in I_1$, $\ps(a)$ is a natural isomorphism
$($i.e., $(F,\ps)$ is $I$-equivariant$)$.
\end{enumerate}
\end{prp}

A $\k$-category $\calA$ is called $\k$-{\em projective} if
$\calA(x,y)$ are projective $\k$-modules for all $x,y \in \calA_0$.
We formulate a categorical version of Keller's lifting theorem
\cite[Theorem 2.1]{Ke2} (in the $\k$-projective case) as follows,
a proof of which is given by B.\ Keller in Appendix.

\begin{thm}[Keller]
\label{keller-lifting}
Let $\calA, \calB \in \kCat$,
and let $T \colon \calA \to \calK(\Prj \calB)$
be a $\k$-functor factoring through the inclusion
$\calK^-(\Prj \calB) \to \calK(\Prj \calB)$.
Set $\nu \colon \calC(\Prj \calB) \to \calK(\Prj \calB)$
and $Q \colon \calK(\Prj \calB) \to \calD(\Mod \calB)$
to be the canonical functors, and put $\pi:= Q\nu$.
Assume that $\calA$ is $\k$-projective
and that $T$ satisfies the {\em Toda condition}:
$$
\calK(\Mod \calB)(T(x), T(y)[n]) = 0, \forall n < 0,
\forall x, y \in \calA_0.
$$
Then the following hold.
\begin{enumerate}
\item[(a)]
There exists a $\k$-functor $B \colon \calA \to \calC(\Prj \calB)$
factoring through the inclusion $\calC^-(\Prj \calB) \to \calC(\Prj \calB)$
and a natural transformation $q \colon T \Longrightarrow \nu B$
such that $q_x \colon T(x) \to \nu B(x)$ are quasi-isomorphisms
for all $x \in \calA_0$.
\item[(b)]
If $B' \colon \calA \to \calC(\Prj \calB)$ is a $\k$-functor 
factoring through the inclusion $\calC^-(\Prj \calB) \to \calC(\Prj \calB)$ and $q' \colon T \Longrightarrow \nu B'$ is a natural transformation 
such that $q'_x \colon T(x) \to \nu B'(x)$ are quasi-isomorphisms
for all $x \in \calA_0$,
then there exists a natural transformation $p \colon B \to B'$
such that $q'=(\nu p) \circ q$ and that $\pi p \colon \pi B \to \pi B'$
is an isomorphism in $\Fun(\calA, \calD(\Mod \calB))%
$.
Further if a natural transformation $p' \colon B \to B'$
has the same property as $p$, then $\nu p = \nu p'$,
i.e., $p$ is unique up to homotopy.
\end{enumerate}

\end{thm}

\begin{dfn}
Let $X\colon I \to \kCat$ be an oplax functor.
\begin{enumerate}
\item
$X$ is called $\k$-{\em projective}
if $X(i)$ are $\k$-projective for all $i \in I_0$.
\item
An oplax subfunctor $\calT$ of of $\Kb(\prj X)$ is called
{\em tilting} if for each $i \in I_0$,
$\calT(i)$ is a tilting subcategory of $\Kb(\prj X(i))$, namely,
\begin{itemize}
\item
$\Kb(\prj X(i))(U, V[n]) = 0$ for all $U, V \in \calT(i)_0$
and $0 \ne n \in \bbZ$; and
\item
the smallest thick subcategory of $\Kb(\prj X(i))$
containing $\calT(i)$ is equal to $\Kb(\prj X(i))$.
\end{itemize}
\item A tilting oplax subfunctor $\calT$ of $\Kb(\prj X)$
with an $I$-equivariant inclusion
$(\si, \ro)\colon \calT \incl \Kb(\prj X)$
is called a {\em tilting oplax functor} for $X$.
\end{enumerate}
\end{dfn}

The following is our main result in this paper that gives
a generalization of the Morita type theorem characterizing derived equivalences of categories by Rickard \cite{Rick} and Keller \cite{Ke1} in our setting.

\begin{thm}
\label{mainthm1}
Let $X, X' \in \Oplax(I, \kCat)$.
Consider the following conditions.
\begin{enumerate}
\item
$X$ and $X'$ are derived equivalent.
\item
$\Kb(\prj X)$ and $\Kb(\prj X')$ are equivalent
in $\Oplax(I, \kTri)$.
\item
There exists a tilting oplax functor $\calT$ for $X$
such that $\calT$ and $X'$ are equivalent in $\Oplax(I, \kCat)$.
\end{enumerate}
Then 
\begin{enumerate}
\item[(a)]
$(1)$ implies $(2)$.
\item[(b)]
$(2)$ implies $(3)$.
\item[(c)]
If $X'$ is $\k$-projective, then $(3)$ implies $(1)$. 
\end{enumerate}
\end{thm}

\begin{proof}
(a) Assume the statement (1).
Then there exists an equivalence
$$
(F, \ps)\colon \calD(\Mod X) \to \calD(\Mod X')
$$
in $\Oplax(I, \kTri)$.
Since $F(i)$ sends compact objects
to compact objects, we have
$F(i)(\perf X(i)) \subseteq \perf X'(i)$ for each $i \in I_0$.
This shows that $(F, \ps)$ induces an equivalence
$\perf X \to \perf X'$.
Hence by Remark \ref{D-form}(2) the statement (2) follows.

(b) Assume the statement (2).
Then we have an equivalence
$$
(F, \ps)\colon \Kb(\prj X') \to \Kb(\prj X)
$$
in $\Oplax(I, \kTri)$.
We define an $I$-equivariant morphism
$$
(H, \ph^H)\colon X' \to \Kb(\prj X')
$$
as follows.
For each $i \in I_0$,
let $H(i)\colon X'(i) \to \Kb(\prj X'(i))$
be the Yoneda embedding, namely it is defined by
sending each
morphism $f\colon x \to y$ in $X'(i)$ to the morphism
$$X'(i)(?, x) \xrightarrow{X'(i)(?, f)} X'(i)(?, y)$$
of complexes concentrated in degree zero.
For each $a\colon i \to j$ in $I$,
let
$$\ph^H(a)x\colon
X'(i)(?,x)\ox_{X'(i)}X'(j)(\text{-},X'(a)(?)) \to
X'(j)(\text{-},X'(a)x)
$$
be the canonical isomorphism.
Then $\ph^H(a):=(\ph^H(a)x)_{x\in X'(i)_0}$ is a
natural isomorphism:
$$
\xymatrix{
X'(i) & \Kb(\prj X'(i))\\
X'(j) & \Kb(\prj X'(j)),
\ar^(.4){H(i)} "1,1";"1,2"
\ar^(.4){H(j)} "2,1";"2,2"
\ar_{X'(a)} "1,1";"2,1"
\ar@{=>}_{\ph^H(a)} "1,2";"2,1"
\ar^{\text{-}\ox_{X'(i)}\ovl{X'(a)}} "1,2";"2,2"
}
$$
and it is easy to check that
$(H, \ph^H):= ((H(i))_{i\in I_0}, (\ph^H(a))_{a\in I_1})$
is a morphism in $\Oplax(I, \kCat)$.

Let $i\in I_0$.
We set $\calT(i)$ to be the full subcategory
of $\Kb(\prj X(i))$ consisting of the objects
$F(i)H(i)x$ with $x \in X'(i)_0$, and
$\si(i)\colon \calT(i) \to \Kb(\prj X(i))$ to be
the inclusion functor.
Then $\calT(i)$ turns out to be a tilting subcategory
of $\Kb(\prj X(i))$ because the full subcategory
of $\Kb(\prj X'(i))$ consisting of the objects
$H(i)x$ with $x \in X'(i)_0$ is tilting and $F(i)$
is a triangle equivalence.
Since $F(i)$ is an equivalence,
$F(i)$ restricts to an equivalence
$R_i\colon X'(i) \to \calT(i)$ with a quasi-inverse $L_i$:
$$
\xymatrix{
\Kb(\prj X'(i)) & \Kb(\prj X(i))\\
X'(i) & \calT(i),
\ar^{F(i)}"1,1";"1,2"
\ar@{-->}_{R_i}"2,1";"2,2"
\ar^{H(i)}"2,1";"1,1"
\ar_{\si(i)}"2,2";"1,2"
}
$$
where we can take $L_i$ as a section of $R_i$ to have
$R_iL_i = \id_{\calT(i)}$.

By Lemma \ref{adjoint-oplax} $(\calT(i))_{i\in I_0}$
extends to an oplax functor  $\calT \in \Oplax(I, \kCat)$
and both $(R_i)_{i\in I_0}$ and $(L_i)_{i\in I_0}$
extend to $I$-equivariant morphisms
$(R,\ph^R)\colon X' \to \calT$ and
$(L,\ph^L)\colon \calT \to X'$, respectively.
Then $(R,\ph^R)$ is an equivalence, and hence
$\calT$ and $X'$ are equivalent in $\Oplax(I, \kCat)$.
We extend the family of inclusions $(\si(i))_{i\in I_0}$
to an $I$-equivariant morphism
$(\si, \ro)\colon \calT \to \Kb(\prj X)$.
Since $F(i)H(i)L_i = \si(i)R_iL_i =\si(i)$ for each $i\in I_0$,
we can define $(\si, \ro)\colon \calT \to \Kb(\prj X)$ by
$(\si, \ro):=(F,\ps)\circ (H, \ph^H)\circ (L, \ph^L)$.
Since $\ph^L(a)$, $\ph^H(a)$ and $\ps(a)$ are isomorphisms,
$\ro(a)$ is also an isomorphism for each $a\in I_1$.
Thus $(\si,\ro)$ is an $I$-equivariant morphism, which
shows that $\calT$ is a tilting oplax functor for $X$.

(c) Assume the statement (3).
Then we have an $I$-equivariant morphism
$$
(E,\ph)\colon X'\to \calT \incl \Kb(\prj X) \incl \calK^-(\Prj X)
$$
such that $E(i)\colon X'(i) \to \calK^-(\Prj X(i))$ is
fully faithful for each $i \in I_0$ and that
$\Kb(\prj X(i))(E(i)x, E(i)y[n])=0$
for each $x, y \in X'(i)$ and each $n\ne 0$.
By Theorem \ref{keller-lifting} (a) there exist
a $\k$-functor $B_i\colon X'(i) \to \calC(\Mod X(i))$
and a quasi-isomorphism $q_i \colon E(i) \Longrightarrow B_i$.
For each $a\colon i \to j$ in $I$ we set
$\chi(a)$ to be the composite
$$
E(i)\ox_{X(i)}\ovl{X(a)} \overset{\ro(a)E(i)}\Longrightarrow \calT(a)E(i) \overset{\ph(a)}\Longrightarrow
E(j)X'(a) = \ovl{X'(a)}\ox_{X'(j)}E(j)
$$
of natural isomorphisms.
Then we have the following diagram with solid arrows.
$$
\xymatrix{
E(i)\ox_{X(i)}\ovl{X(a)} & \ovl{X'(a)}\ox_{X'(j)}E(j)\\
\nu B_i \ox_{X(i)}\ovl{X(a)} & \ovl{X'(a)}\ox_{X'(j)}\nu B_j
\ar@{=>}^{\chi(a)} "1,1";"1,2"
\ar@{=>}_{q_i\ox_{X(i)}\ovl{X(a)}}"1,1";"2,1"
\ar@{=>}^{\ovl{X'(a)}\ox_{X'(j)}q_j} "1,2";"2,2"
\ar@{==>} "2,1";"2,2"
}
$$
We show that this is completed to a commutative diagram.
Let $x \in X'(i)_0$.
\setcounter{clm}{0}
\begin{clm}
$\chi(a)x$ is a quasi-isomorphism.
\end{clm}
Indeed, $E(i)x \in \Kb(\prj X(i))$ implies
$E(i)x\ox_{X(i)}\ovl{X(a)}\in \Kb(\prj X(j))$,
and hence $\chi(a)x$ is given by a genuine morphism
and is an isomorphism in $\Kb(\prj X(j))$.
\begin{clm}
$\ovl{X'(a)}\ox_{X'(j)}q_j(x)$ is a quasi-isomorphism.
\end{clm}
This is obvious because $q_j(x)$ is a quasi-isomorphism in
$\Kb(\prj X(j))$ and $\ovl{X'(a)}_{X'(j)}$ is projective.
\begin{clm}
$q_i(x)\ox_{X(i)}\ovl{X(a)}$ is a quasi-isomorphism.
\end{clm}
Indeed, since $q_i(x)$ is a quasi-isomorphism in
$\calK^-(\Prj X(i))$ the mapping cone $C(q_i(x))$ is
acyclic, and hence so is $C(q_i(x))\ox_{X(i)}\ovl{X(a)}$,
from which the claim follows.
\begin{clm}
$E(i)\ox_{X(i)}\ovl{X(a)}$ satisfies the Toda condition.
\end{clm}
Indeed, let $y$ be another object of $X'(i)$ and $n \ne 0$.
Then
$$
\begin{gathered}
\calK^-(\Prj X(j))\left(E(i)x\ox_{X(i)}\ovl{X(a)},
\left(E(i)y \ox_{X(i)}\ovl{X(a)}\right)[n]\right)\\
\iso \calK^-(\Prj X(j))(E(j)(X'(a)x), E(j)(X'(a)y)[n])=0.
\end{gathered}
$$

By Theorem \ref{keller-lifting} (b), it follows from these claims
that there exists a natural transformation
$\ps(a)\colon B_i\ox_{X(i)}\ovl{X(a)}
\Longrightarrow \ovl{X'(a)}\ox_{X'(j)}B_j$
such that $\nu\ps(a)$ completes the commutative diagram above and
$\pi\ps(a)$ is an isomorphism.
Thus we have the following diagram
$$
\xymatrix@C=5pc{
\calD(\Mod X'(i)) &\calD(\Mod X(i))\\
\calD(\Mod X'(j)) & \calD(\Mod X(j)),
\ar^{\ovl{B}(i)}"1,1";"1,2"
\ar_{\blank\Ltimes_{X'(i)}\ovl{X'(a)}}"1,1";"2,1"
\ar^{\blank\Ltimes_{X(i)}\ovl{X(a)}} "1,2";"2,2"
\ar_{\ovl{B}(j)}"2,1";"2,2"
\ar@{=>}_{\ovl{\ps}(a)}"1,2";"2,1"
}
$$
where $\ovl{B}(i):=\blank\Ltimes_{X'(i)}\pi B_i$ are triangle equivalences
because $\pi B_i$ are tilting bimodule complexes,
and $\ovl{\ps}(a):=\blank\Ltimes_{X'(i)}\pi\ps(a)$ are natural isomorphisms.
We set $\ovl{B}:=(\ovl{B}(i))_{i\in I_0}$ and
$\ovl{\ps}:= (\ovl{\ps}(a))_{a\in I_1}$.
It remains to show that the pair
$$
(\ovl{B}, \ovl{\ps}) \colon \calD(\Mod X') \to \calD(\Mod X)
$$
is a 1-morphism in $\Oplax(I, \kTri)$
because when this is proved, $(\ovl{B}, \ovl{\ps})$ becomes an equivalence in $\Oplax(I, \kTri)$
by Proposition \ref{der-eq-criterion},
and we see that $X$ and $X'$ are
derived equivalent.

It is enough to show that the diagram
\begin{equation}
\label{B-a}
\vcenter{
\xymatrix@C=5pc{
\pi B_i \ox_{X(i)}\ovl{X(\id_i)}&\pi B_iX'(\id_i)\\
B_i & B_i
\ar@{=>}^{\ps(\id_i)}"1,1";"1,2"
\ar@{=}"2,1";"2,2"
\ar@{=>}_{\ovl{\et}_iB_i}"1,1";"2,1"
\ar@{=>}^{B_i\et_i}"1,2";"2,2"
}
}
\end{equation}
for each $i \in I_0$ and the diagram
\begin{equation}
\label{B-b}
\vcenter{
\xymatrix{
\pi B_i \ox_{X(i)}\ovl{X(ba)} &
\pi B_i\ox_{X(i)}\ovl{X(a)}\ox_{X(j)}\ovl{X(b)} &
\ovl{X'(a)}\ox_{X'(j)}\pi B_j \ox_{X(j)} \ovl{X(b)}\\
\ovl{X'(ba)}\ox_{X'(k)}\pi B_k &&
\ovl{X'(a)}\ox_{X'(j)}\ovl{X'(b)}\ox_{X'(k)}\pi B_k
\ar@{=>}^-{\pi B_i\ovl{\th}_{b,a}}"1,1";"1,2"
\ar@{=>}^{\sim \ps(a)\sim} "1,2";"1,3"
\ar@{=>}_{\sim \ovl{\th}_{b,a}\sim}"2,1";"2,3"
\ar@{=>}_{\ps(ba)}"1,1";"2,1"
\ar@{=>}^{\sim \ps(b) \sim} "1,3";"2,3"
}
}
\end{equation}
for each $i\ya{a}j \ya{b} k$ in $I$ commute,
where we put $X=(X, \et, \th)$
and $X'=(X', \et', \th')$, and $\ovl{\et}_i$,
$\ovl{\th}_{b,a}$ denote the morphisms
induced by $\et_i$, $\th_{b,a}$, respectively.

The commutativity of the diagram \eqref{B-a} follows from
the following commutative diagram by using the fact that
$q_i\ox_{X(i)}\ovl{X(\id_i)}$ is a quasi-isomorphism:
$$
\xymatrix{
\nu B_i\ox_{X(i)}\ovl{X(\id_i)}&&& \nu B_iX'(\id_i)\\
&T(i)\ox_{X(i)}\ovl{X(\id_i)}& T(i)X'(\id_i)\\
&T(i) &T(i)\\
\nu B_i &&& \nu B_i.
\ar@{=>}^{\nu\ps(\id_i)}"1,1";"1,4"
\ar@{=>}"2,2";"2,3"
\ar@{=}"3,2";"3,3"
\ar@{=}"4,1";"4,4"
\ar@{=>}"1,1";"4,1"
\ar@{=>}"2,2";"3,2"
\ar@{=>}^{T(i)\et_i}"2,3";"3,3"
\ar@{=>}^{\nu B_i\et_i}"1,4";"4,4"
\ar@{=>}_{q_i\ox_{X(i)}\ovl{X(\id_i)}}"2,2";"1,1"
\ar@{=>}_{q_iX'(\id_i)}"2,3";"1,4"
\ar@{=>}^{q_i}"3,2";"4,1"
\ar@{=>}_{q_i}"3,3";"4,4"
}
$$
The commutativity of the diagram \eqref{B-b} follows from
the following commutative diagram by using the fact that
$T(ba)q_i$ is a quasi-isomorphism:
$$
\xymatrix{
T(ba)\nu B_i &&T(b)T(a)\nu B_i && T(b)\nu B_j X'(a)\\
&T(ba)E(i)& T(b)T(a)E(i) & T(b)E(j)X'(a)\\
&E(k)X'(ba) && E(k)X'(b)X'(a)\\
\nu B_kX'(ba)&&&&\nu B_kX'(b)X'(a),
\ar@{=>}^{\ovl{\th}_{b,a}\nu B_i}"1,1";"1,3"
\ar@{=>}^{T(b)\ps(a)}"1,3";"1,5"
\ar@{=>}_{\ovl{\th}_{b,a}E(i)}"2,2";"2,3"
\ar@{=>}_{T(b)\chi(a)}"2,3";"2,4"
\ar@{=>}_{E(k)\th'_{b,a}}"3,2";"3,4"
\ar@{=>}^{\nu B_k \th'_{b,a}}"4,1";"4,5"
\ar@{=>}_{\ps(ba)}"1,1";"4,1"
\ar@{=>}_{\chi(ba)}"2,2";"3,2"
\ar@{=>}_{T(b)T(a)q_i}"2,3";"1,3"
\ar@{=>}^{\chi(b)X'(a)}"2,4";"3,4"
\ar@{=>}_{\ps(b)X'(a)}"1,5";"4,5"
\ar@{=>}_{T(ba)q_i}"2,2";"1,1"
\ar@{=>}^{T(b)q_jX'(a)}"2,4";"1,5"
\ar@{=>}^{q_kX'(ba)}"3,2";"4,1"
\ar@{=>}_{q_kX'(b)X'(a)}"3,4";"4,5"
}
$$
where we regard
$T(a)=\blank\ox_{X(i)}\ovl{X(a)}\colon T(i)\to T(j)$.
\end{proof}

\begin{dfn}
Regard a group $G$ as a category with a unique object $*$.
Then a {\em $\k$-category with a pseudo-action} of $G$ is a pair $(\calC, X)$
of a $\k$-category $\calC$ and a pseudo-functor $X \colon G \to \kCat$ with $\calC = X(*)$.
\end{dfn}

As a special case of  Theorem \ref{mainthm1} we obtain the following.

\begin{cor}
Let $G$ be a group, $(\calC, X)$ and $(\calC', X')$ $\k$-categories with pseudo-actions of $G$.
Assume that $\k$ is a field.
Then the following are equivalent.
\begin{enumerate}
\item $(\calC, X)$ and $(\calC', X')$ are derived equivalent.
\item There exists a $G$-equivariant tilting subcategory $\calT$ with a pseudo-action of $G$
in $\Kb(\prj(\calC, X))$
such that $(\calC', X')$ is equivalent to $\calT$ in the $2$-category $\Oplax(G, \kCat)$.
\end{enumerate}
\end{cor}

\begin{exm}
Consider $\bbZ$ as an additive group.
For a $\k$-algebra $B$ and an automorphism $\la$ of $B$ denote by $(B, \la)$ the category $B$
with a $\bbZ$-action defined by sending $1$ to $\la$, and by $\hat{B}$ the repetitive category of $B$.
If $A$ and $A'$ are derived equivalent algebras, then the categories
$(\hat{A}, \nu^n)$ and $(\hat{A'}, \nu'^n)$ with $\bbZ$-actions
are derived equivalent for all $n \in \bbN$, where $\nu$ (resp.\ $\nu'$) are
the Nakayama automorphism of $\hat{A}$ (resp.\ $\hat{A'}$).
\end{exm}

By applying Theorem \ref{mainthm1} to the free category of the quiver $1 \ya{a} 2$
we obtain the following.

\begin{cor}
Let $\la \colon A \to B$ and $\la' \colon A' \to B'$ be morphisms of $\k$-algebras,
by which we regard $B$, $B'$ as a left $A$-module and a left $A'$-module, respectively.
Assume that $\k$ is a field.
Then the following are equivalent.
\begin{enumerate}
\item There exist equivalences $F$, $G$ of triangulated categories such that the following diagram is commutative up to natural isomorphisms
$$
\xymatrix{
\calD(\Mod A) & \calD(\Mod A')\\
\calD(\Mod B) & \calD(\Mod B').
\ar^{F}"1,1";"1,2"
\ar_{G}"2,1";"2,2"
\ar_{\blank\Ltimes B}"1,1";"2,1"
\ar^{\blank\Ltimes B'}"1,2";"2,2"
}
$$
\item There exist a tilting complex $T$ for $A$ with $T \otimes_A B$ a tilting complex for $B$,
$\k$-algebra isomorphisms $\al$, $\be$ and a $\k$-algebra morphism $\mu$
such that the following diagram is commutative up to natural isomorphism
$$
\xymatrix{
A' & \End_{\Kb(\prj A)}(T) & \Kb(\prj A)\\
B' & \End_{\Kb(\prj B)}(T\otimes_AB) & \Kb(\prj B).
\ar^-{\al}"1,1";"1,2"
\ar^--{\be}"2,1";"2,2"
\ar@{^{(}->}"1,2";"1,3"
\ar@{^{(}->}"2,2";"2,3"
\ar_{\la'}"1,1";"2,1"
\ar^{\mu}"1,2";"2,2"
\ar^{\blank\otimes_AB}"1,3";"2,3"
}
$$
\end{enumerate}

\end{cor}

\section{Appendix: A categorical version of Keller's lifting theorem}
\label{appx}

\newcommand{\cb}{{\mathcal B}}

\newcommand{\HOM}{\opname{Hom^\bullet}}
\newcommand{\cc}{{\mathcal C}}
\newcommand{\ck}{{\mathcal K}}
\newcommand{\ce}{{\mathcal E}}
\newcommand{\ca}{{\mathcal A}}
\newcommand{\ko}{\: , \;}
\newcommand{\fk}{\k}
\newcommand{\bp}{\mathbf{p}}
\newcommand{\ol}{\overline}
\newcommand{\cd}{{\mathcal D}}

\newcommand{\ra}{\rightarrow}

\newcommand{\opname}[1]{\operatorname{\mathsf{#1}}}

\renewcommand{\Hom}{\opname{Hom}}

\newcommand{\opnamestar}[1]{\operatorname*{\mathsf{#1}}}
\renewcommand{\mod}{\opname{mod}\nolimits}
\renewcommand{\Mod}{\opname{Mod}\nolimits}
\newcommand{\Add}{\opname{Add}\nolimits}
\newcommand{\Sum}{{\mbox{Sum}}}
\newcommand{\Gen}{{\mbox{Gen}}}
\newcommand{\Tria}{{\mbox{Tria}}}
\newcommand{\Susp}{{\mbox{Susp}}}
\newcommand{\aisle}{{\mbox{aisle}}}
\newcommand{\cell}{{\mbox{cell}}}
\newcommand{\proj}{\opname{proj}\nolimits}
\newcommand{\coh}{\opname{coh}\nolimits}

\newcommand{\R}{\mathbf{R}}
\renewcommand{\L}{\mathbf{L}}
\newcommand{\ten}{\otimes}
\newcommand{\lten}{\otimes^{\mathsf{L}}}
\newcommand{\tp}[1]{^{\ten #1}}
\newcommand{\coker}{\opname{coker}\nolimits}
\newcommand{\im}{\opname{im}\nolimits}
\renewcommand{\ker}{\opname{ker}\nolimits}
\renewcommand{\colim}{\opname{colim}\nolimits}
\newcommand{\lid}{\varinjlim}
\newcommand{\lii}{\varprojlim}
\newcommand{\Mcolim}{\opname{Mcolim}}
\newcommand{\Mlim}{\opname{Mlim}}
\renewcommand{\can}{\opname{can}}
\renewcommand{\incl}{\opname{incl}}
\renewcommand{\Iso}{\opname{Iso}}
\newcommand{\Quot}{\opname{Quot}}

%
%
\newcommand{\cca}{\cc\ca}
\newcommand{\ccb}{\cc^b}
\newcommand{\acb}{\ca c^b}
\newcommand{\cone}{\opname{Cone}\nolimits}
\newcommand{\con}{\opname{con}\nolimits}
\newcommand{\vecbundle}{\opname{vec}\nolimits}
\newcommand{\spec}{\opname{Spec}\nolimits}

\newcommand{\uni}{\mathbf{U}}
\newcommand{\ex}{\opname{Ex}}
\newcommand{\adm}{\opname{Adm}}
\newcommand{\sat}{\opname{Sat}}
\newcommand{\filt}{\opname{filt}}
\newcommand{\tria}{\opname{tria}}
\newcommand{\tw}{\opname{tw}}
\newcommand{\twst}{\opname{twst}}

\newcommand{\mix}{\cm ix\,}
\newcommand{\dmix}{\cd\cm ix\,}
\newcommand{\mormix}{\cm or \cm ix\,}
\newcommand{\dmormix}{\cd\mormix}
\renewcommand{\Fun}{\opname{Fun}}
\newcommand{\fun}{\opname{fun}}
\renewcommand{\com}{\cc}

\newcommand{\Dif}{\opname{Dif}}
\newcommand{\dgfree}{\opname{dgfree}}
\newcommand{\Rep}{\opname{Rep}}
\renewcommand{\rep}{\opname{rep}}
\newcommand{\fil}{\opname{Fil}}
\newcommand{\per}{\opname{per}}

\newcommand{\hca}{\opname{Hca}}

\newcommand{\tr}{\opname{tr}}
\newcommand{\tot}{\opname{tot}}
\newcommand{\Tot}{\opname{Tot}}
\newcommand{\hTot}{\Hat{\Tot}}
%
%
\newcommand{\centeps}[1]{\begin{array}{c} \epsfbox{#1} \end{array}}
\newcommand{\centdoseps}[2]{\begin{array}{cc}\epsfbox{#1} & \epsfbox{#2}\end{array}}
\setlength{\unitlength}{0.25cm}
\newcommand{\lblarge}[3]{\put(#1,#2){\makebox(3,2){$ #3 $}}}
\newcommand{\lb}[3]{\put(#1,#2){\makebox(2,2){$ #3 $}}}
\newcommand{\lbt}[3]{\put(#1,#2){\makebox(0,0)[t]{$ #3 $}}}
\newcommand{\lbb}[3]{\put(#1,#2){\makebox(0,0)[b]{$ #3 $}}}
\newcommand{\lbtl}[3]{\put(#1,#2){\makebox(0,0)[tl]{$ #3 $}}}
\newcommand{\lbtr}[3]{\put(#1,#2){\makebox(0,0)[tr]{$ #3 $}}}
\newcommand{\lbbl}[3]{\put(#1,#2){\makebox(0,0)[bl]{$ #3 $}}}
\newcommand{\lbbr}[3]{\put(#1,#2){\makebox(0,0)[br]{$ #3 $}}}
\newcommand{\picarrow}[5]{\put(#1,#2){\vector(#3,#4){#5}}}
\newcommand{\lin}[5]{\put(#1,#2){\line(#3,#4){#5}}}
\newcommand{\lbold}[3]{\put(#1,#2){\makebox(1,0.5){$ #3 $}}}
\renewcommand{\dot}[2]{\put(#1,#2){\circle*{0.5}}}

\newcommand{\comment}[1]{}

\hyphenation{Grothen-dieck}

\def\Z{\bbZ}

We prove Theorem~\ref{keller-lifting}. We will prove existence of the lifting following section~9 of
\cite{Ke1}.  Uniqueness follows easily.
We use the notations of the main body of the paper. Moreover,
for two complexes $L$ and $M$ of $\cb$-modules, we write
$\HOM_\cb(L,M)$ for the complex whose $n$th component is formed
by the morphisms $f:L \to M$ homogeneous of degree $n$ between the
$\Z$-graded objects underlying $L$ and $M$ and whose differential
takes $f$ to $d(f)=d_M \circ f - (-1)^n f\circ d_L$. The class of
complexes of $\cb$-modules endowed with the assignment
\[
(L,M) \mapsto \HOM_\cb(L,M)
\]
naturally becomes a dg category \cite{Ke1} \cite{Keller06d} denoted
by $\cc_{dg}(\Mod\cb)$ and we have
\[
Z^0 \cc_{dg}(\Mod \cb) = \cc(\Mod\cb) \quad\mbox{and} \quad
H^0 \cc_{dg}(\Mod \cb) = \ck(\Mod\cb).
\]

\subsection{Existence}
Let $\ce$ be the dg endomorphism category of $T$: Its objects are those
of $\ca$ and for two objects $x$, $y$ of $\ca$, we put
\[
\ce(x,y) = \HOM_\cb(T(x),T(y)).
\]
Thus, for $n\in\Z$, we have
\[
H^n \ce(x,y) = \Hom_{\ck(\Prj\cb)}(T(x),T(y)[n])
\]
and this group vanishes for $n<0$ by the Toda condition. Let $\tau_{\leq 0}\ce$
denote the dg subcategory of $\ce$ with the same objects and with the
morphism complexes
\[
(\tau_{\leq 0}\ce)(x,y) = \tau_{\leq 0}(\ce(x,y)).
\]
By the Toda condition, the projection $\tau_{\leq 0}\ce \to H^0\ce$ is a
quasi-isomorphism (i.e.\ a dg functor which induces a bijection on the
objects and quasi-isomorphisms in the morphism complexes). Since we
have
\[
(H^0\ce)(x,y) = \Hom_{\ck(\Prj \cb)}(T(x), T(y)) \ko
\]
the functor $T$ yields a functor $F: \ca \to H^0\ce$ which is the identity
on the objects and given by $T$ on the morphisms. Thus, we obtain
a chain of dg functors
\[
\xymatrix{
\ca \ar[r]^-F & H^0\ce & \tau_{\leq 0} \ce\mbox{ } \ar@{->>}[l]_{qis} \ar@{>->}[r] & \ce \ar[r] &
\cc_{dg}(\Prj \cb).
}
\]
To `invert' the quasi-isomorphism, we now temporarily pass from functors
to bimodules: Let $X_1$ be the $\ca^{op}\ten_\fk H^0\ce$--module given by
\[
X_1(x,y) = (H^0\ce)(x,Fy) .
\]
Since $\ca$ is projective over $\fk$, the induced dg functor
\[
\xymatrix{ \ca^{op} \ten_\fk H^0\ce & \ca^{op} \ten_\fk \tau_{\leq 0} \ce \ar[l] }
\]
is a quasi-isomorphism. The restriction of $X_1$ along this quasi-isomorphism
is still denoted by $X_1$. Now let us denote by $T'$ the $\ce^{op}\ten_\fk \cb$--module
given by
\[
T'(?,y) = T(y).
\]
We put
\[
X_2 = X_1 \lten_{\tau_{\leq 0} \ce} T' = (\bp X_1) \ten_{\tau_{\leq 0} \ce} T' \ko
\]
where $\bp X_1$ is a {\em cofibrant resolution} (\cite[Section 2.12]{Ke-Ya}) of the
$\ca^{op}\ten_\fk \tau_{\leq 0}\ce$--module $X_1$. Notice that $X_1$ is
right bounded so that $\bp X_1$ may be chosen right bounded. Since
$T'$ is also right bounded, the tensor product $X_1$ is right bounded.
Moreover, since $\ca(x,y)$ is $\fk$-projective for all objects $x$, $y$ of $\ca$,
cofibrant modules over $\ca^{op}\ten_\fk \tau_{\leq 0}\ce$ tensored
by right cofibrant $\tau_{\leq 0}\ce \ten_\fk \cb$-modules are
cofibrant over $\cb$. Thus, for each object $x$ of $\ca$, the complex
$X_2(?,x)$ is a right bounded complex of projective $\cb$-modules.
We define $B: \ca \to \cc(\Prj \cb)$ by
\[
B(x) = X_2(?,x).
\]
Now let us construct the natural transformation $q: T \to \nu B$ of
functors from $\ca$ to $\ck(\Prj \cb)$. For an object $x$ of $\ca$, the
object $X_1(?,x)$ is isomorphic to $(H^0\ce)(?,x)$ and we have
two quasi-isomorphisms with cofibrant $\tau_{\leq 0}\ce$--modules
given by
\[
\xymatrix{
 (\tau_{\leq 0} \ce)(?,x)  \ar[r]  & (H^0\ce)(?,x) & (\bp X_1)(?,x) \ar[l] }
\]
Thus, there is a unique morphism $ (\tau_{\leq 0} \ce)(?,x)\to (\bp X_1)(?,x)$
in the homotopy category of
$\tau_{\leq 0}\ce$--modules which makes the following triangle commutative
\[
\xymatrix{  (\tau_{\leq 0} \ce)(?,x) \ar[rd] \ar[rr] &  &  (\bp X_1)(?,x) \ar[ld] \\
 & (H^0\ce)(?,x) & }
\]
By tensoring this morphism with $T'$ we obtain a morphism
\[
T(x) = (\tau_{\leq 0} \ce)(?,x) \ten_{\tau_{\leq 0} \ce} T' \to
(\bp X_1(?,x))\ten_{\tau_{\leq 0}\ce} T' = X_2(?,x) = \nu B(x) \ko
\]
which we use to define $q_x$. It is straightforward to check that the
$q_x$ do yield a natural transformation $q: T \to \nu B$.

\subsection{Uniqueness} Suppose we are given a morphism
$q': T \to \nu B'$ as in part (b) of Theorem \ref{keller-lifting}. Then
clearly there exists a unique natural transformation
$\ol{p}: \pi B \to \pi B'$ of functors from $\ca$ to
$\calD(\Mod\cb)$ such that $Qq'=\ol{p}\circ Qq$.
Since for each $x \in \mathcal{A}$,
$\ovl{p}x: B(x) \to B'(x)$
is in $\calD^{-}(\Prj \mathcal{B})$,
$\ol{p}$ lifts
to a morphism $p: B \to B'$ of functors  from
$\ca$ to $\cc^-(\Prj \cb)$.
Thus we have $\pi p = \ovl{p}$.
Then $Qq'=Q(\nu p \circ q)$ and we have $q'=\nu p \circ q$
because $Q$ is fully faithful on $\calK^-(\Prj \cb)$.
If $q'=\nu p' \circ q$ for some $p': B \to B'$,
then $Q\nu p = Qq'(Qq)\inv = Q\nu p'$ shows $\nu p = \nu p'$
by the same reason. \qed


\begin{thebibliography}{99}

\bibitem{Asa97} Asashiba, H.:
{\it A covering technique for derived equivalence},
J.\ Algebra., {\bf 191} (1997), 382--415.

\bibitem{Asa99} \bysame:
{\it The derived equivalence classification of representation-finite
selfinjective algebras}, J.\ Algebra, {\bf 214} (1999), 182--221.

\bibitem{Asa02} \bysame:
{\it Derived and stable equivalence classification of twisted multifold extensions of piecewise hereditary algebras of tree type},
J.\ Algebra  {\bf 249} (2002), 345--376.

\bibitem{Asa11} \bysame :
{\it A generalization of Gabriel's Galois covering functors and derived equivalences}, J.\ Algebra {\bf 334} (2011), 109--149. 

\bibitem{Asa} \bysame :
{\it Gluing derived equivalences together}, preprint, arXiv:1204.0196.

\bibitem{Gab} Gabriel, P.:
{\it The universal cover of a representation-finite algebra},
in: Lecture Notes in Math., vol. {\bf 903}, Springer-Verlag,
Berlin/New York, 1981, 68--105.

\bibitem{Groth} Grothendieck, A.:
{Rev\^etements \'etales et groupe fondamental},
Springer-Verlag, Berlin, 1971.
S\'eminaire de G\'eom\'etrie Alg\'ebrique du Bois Marie 1960--1961 (SGA 1),
Dirig\'e par Alexandre Grothendieck.
Augment\'e de deux expos\'es de M. Raynaud, Lecture Notes in Mathematics, Vol. {\bf 224}.

\bibitem{Ke1} Keller, B.:
{\it Deriving DG categories},
Ann. scient. \'{E}c. Norm. Sup., 4$^e$ s\'{e}rie,
t. {\bf 27}, 1994, 63--102.

\bibitem{Ke2} \bysame :
{\it Bimodule complexes via strong homotopy actions},
Algebras and Representation theory, Vol. {\bf 3}, 2000, 357-376.
\newblock Special issue dedicated to Klaus Roggenkamp on the occasion of his
  60th birthday.

\bibitem{Keller06d}
\bysame :
\newblock On differential graded categories.
\newblock In {\em International Congress of Mathematicians. Vol. II}, pages
  151--190. Eur. Math. Soc., Z\"urich, 2006.
  
 \bibitem{Ke-Ya}
 Keller, B.\ and Yang, D.:
 {\em Derived equivalences from mutations of quivers with potential},
Advances in Math.\ {\bf 226} (2011) 2118--2168.

\bibitem{Rick} Rickard, J.:
{\it Morita theory for derived categories}, J. London Math. Soc., {\bf 39}
1989, 436--456.

\bibitem{Str72}
Street, R.:
{\em Two constructions on lax functors},
Cahiers Topologie G{\'e}om.\ Diff{\'e}rentielle , {\bf 13} (1972), 217--264.

\bibitem{Tam} Tamaki, D.:
{\it Grothendieck constructions for enriched categories}, preprint, arXiv:0907.0061.

\end{thebibliography}
\end{document}